\documentclass[12pt,a4paper,oneside]{amsart}
\usepackage[english]{babel}
\usepackage[utf8]{inputenc}
\usepackage{amsmath,amssymb}
\usepackage{amsthm,mathtools}
\usepackage[foot]{amsaddr}
\usepackage{graphicx}
\usepackage{bbm}
\usepackage{geometry}
\usepackage{tikz}
\usepackage{tikz-cd}
\usetikzlibrary{intersections,decorations.markings,calc,angles}
\usepackage{mathrsfs}
\usepackage[colorinlistoftodos]{todonotes}
\usepackage{enumitem}
\usepackage{mycommands}
\usepackage{hyperref}
\usepackage{verbatim}

\theoremstyle{plain}
\newtheorem{thm}{Theorem}[section]
\newtheorem{lem}[thm]{Lemma}
\newtheorem{prop}[thm]{Proposition}

\newtheorem{cor}[thm]{Corollary}

\theoremstyle{definition}
\newtheorem{defn}[thm]{Definition}
\theoremstyle{remark}
\newtheorem*{rmk}{Remark}
\newtheorem*{rmks}{Remarks}

\usepackage{cleveref}
\crefname{section}{§}{§§}
\Crefname{section}{§}{§§}

\title{Meromorphic quasi-modular forms and their $L$-functions}
\author{\small Weijia Wang}
\author{\small Hao Zhang}
\address[Weijia Wang]{Yanqi Lake Beijing Institute of Mathematical Sciences and Applications $\&$ Yau Mathematical Sciences Center, Tsinghua University, Beijing 101408, P. R. China}
\address[Hao Zhang]{School of Mathematics, Hunan University, Changsha 410082, P. R. China}
\email{weijiawang@tsinghua.edu.cn}
\email{zhanghaomath@hnu.edu.cn}
\date{}

\begin{document}

\maketitle
\begin{abstract}
    We investigate the meromorphic quasi-modular forms and their $L$-functions. We study the space of meromorphic quasi-modular forms. Then we define their $L$-functions by using the technique of regularized integral. Moreover, we give an explicit formula for the $L$-functions. As an application, we obtain some vanishing results of special $L$-values of meromorphic quasi-modular forms.
\end{abstract}

\section{Introduction}


Given  a holomorphic cusp form $f(\tau)=\sum_{n>0}a_f(n)q^n$ of weight $k$ on $\SL_2(\mbz)$, where $q=e^{2\pi i\tau}$, the Dirichlet $L$-function associated to $f$ is the series
\[L(f,s)=\sum_{n\geq 1}\frac{a_f(n)}{n^s}=\frac{(2\pi)^s}{\Gamma(s)}\int_0^{\infty}f(it)t^{s-1}dt.\]
It is well-known that $L(f,s)$ satisfies a functional equation under $s\mapsto k-s$ and admits a meromorphic continuation to the complex plane. We can extend the $L$-function to holomorphic quasi-modular forms by Fourier expansions naturally. Due to the work of Kaneko and Zagier~\cite{KZ95}, holomorphic quasi-modular form is always a linear combination of iterated derivatives of modular forms and $E_2$. So their $L$-function comes from shifts of the $L$-functions of  modular forms and the $L$-function of $E_2$. One can also find more details of holomorphic quasi-modular form and their $L$-functions in \cite{BD19, CS179, Zagier08}.


However, everything goes differently if we consider meromorphic quasi-modular forms. If $f$ is a weakly holomorphic modular form, then it has exponential growth at infinity. So in this case, the Dirichlet series $L(f,s)$ associated to the Fourier coefficients of $f$ never converges. To overcome this  problem, we need to introduce the regularized integrals. The regularized integrals and $L$-functions of weakly holomorphic modular forms have been studied in \cite{BFK14}. L\"obrich and Schwagenscheidt studied the $L$-values of some special meromorphic modular forms using Cauchy principal valued integrals in \cite{LS22}. McGady defined and investigated the $L$-functions for meromorphic modular forms which are holomorphic at infinity in \cite{Mc19}.

In this paper, we will study the structure of meromorphic quasi-modular forms and their Dirichlet $L$-functions. Unlike the classical case, the depth of meromorphic quasi-modular forms could be larger than $k/2$, so a meromorphic quasi-modular form may not become a linear combination of iterated derivatives of meromorphic modular forms. In fact, we have the following structure theorem.
\begin{thm}\label{thm:decomposition}
We have the following decomposition of $\C$-vector space of meromorphic quasi-modular forms
$$
\qm_k^{*}=\bigg(\bigoplus_{l=0}^{\frac{k}{2}-1}D^l\mcm_{k-2l}^{*}\bigg)\bigoplus \bigg(\bigoplus_{l=0}^{\frac{k}{2}-1}D^l\qm_{k-2l}^{*,\,k-2l-1}\bigg)\bigoplus\bigg(\bigoplus_{l=k}^{\infty}D^l\mcm_{k-2l}^{*}\bigg).
$$
\end{thm}

We next generalize the Rankin--Cohen bracket to meromorphic modular forms. Cohen \cite{Coh75} proved that the Rankin--Cohen bracket of two holomorphic modular forms is again a modular form. We prove that the Rankin--Cohen bracket of two meromorphic modular forms is also a meromorphic modular form. Together with Theorem \ref{thm:decomposition}, we prove 

\begin{thm}~\label{thm:fg}
Let $f$, $g$ be two meoromorphic quasi-modular forms in the space $\qm_{+}^{*}$ or $\qm_{-}^{*}$ of weight $k$, $l$ and depth $s$, $t$  respectively. Then
$$fg\in \mcm_{k+l}^{*}\bigoplus\bigg(\bigoplus_{i=0}^{s+t}D^{s+t-i}\mcs_{k+l-2s-2t+2i}^{*}\bigg),$$
where $\mcs^{*}$ is the subalgebra of $\mcm^{*}$ consisting of meoromorphic cusp forms, i.e. those with zero constant terms in their Fourier expansions.
\end{thm}

To state our main result, we need to generalize the regularized integral to meromorphic quasi-modular form $f$. For the meromorphic quasi-modular form $f$, we consider its regularized integral 
\[\Lambda(f,s)=\int_0^{\infty,\ast}f(it)t^{s-1}dt.\]
We will give an explicit formula for $\Lambda(f,s)$ in Theorem \ref{thm:lfunction}. The $L$-function $\Lambda(f,s)$ admits some classical properties as usual. More precisely, we have 
\begin{thm}\label{thm:lambdaf}
Let $f\in\qm_k^{\mero,\,p}$ be a meromorphic quasi-modular form of weight $k$ and depth $p$. Let $f_0,\cdots,f_p$ be the component functions corresponding to $f$. Then we have
\begin{enumerate}[label=(\roman*)]
    \item  The complete $L$-function $\Lambda(f,s)$ extends to a meromorphic function for all $s\in\C$ with the
    functional equation
    \begin{equation*}
    \Lambda(f,s)=\sum_{r=0}^{p}i^{k-r}\Lambda(f_{r},k-r-s). 
    \end{equation*}
    \item The complete $L$-function $\Lambda(f,s)$ has only possibly simple poles either at $s=0$ or at all integers within $k-p\leq s\leq k$,  Moreover, the residue of $\Lambda(f,s)$ at an integer $n$ is
    $$  \Res_{s=n}\, \Lambda(f,s)=
    \begin{cases*}
     -a_{f}(0)+a_{f_k}(0) & if $n=0$\\
     i^{n}\,a_{f_{k-n}}(0)& if $n\neq0$ and $k-p\leq n\leq k$
    \end{cases*}
    .$$

     \item  In general, the $L$-functions of component functions $\Lambda(f_m,s)$ satisfy the functional equations
\begin{equation*}
   \Lambda(f_m,s)=\sum_{r=0}^{p-m}i^{k-2m-r}\binom{m+r}{r}\Lambda(f_{m+r},k-2m-r-s). 
\end{equation*}

\end{enumerate}

\end{thm}



The paper is organized as follows. In Section \ref{sec:2}, we introduce the notations and basic properties of meromorphic quasi-modular forms. In Section \ref{sec:3}, we recall the Maass--Shimura derivative and Serre derivative. In Section \ref{sec:4} and Section \ref{sec:5}, we prove the structure theorem of meromorphic quasi-modular forms and generalize the Rankin--Cohen bracket to meromorphic quasi-modular forms. In Section \ref{sec:6} and Section \ref{sec:7}, we introduce the regularized integral for meromorphic functions. In Section \ref{sec:8}, we define the  $L$-function of meromorphic quasi-modular form through regularized integral and give an explicit formula for the $L$-function. Finally, in Section \ref{sec:9}, we give some special values of $L$-functions.


\section{meromorphic quasi-modular forms}\label{sec:2}
Let $\Hcal=\{\tau\in\C\,|\,\IM(\tau)>0\}$ be the Poincaré upper half-plane. A \emph{meromorphic modular form} of weight $k\in\Z$ is a meromorphic function on $\Hcal$ which satisfies
$$
f|_{k}\gamma(\tau)\coloneqq (c\tau+d)^{-k}f\left(\frac{a\tau+b}{c\tau+d}\right)=f(\tau),\quad \text{where }\gamma=\begin{pmatrix*}a&b\\c&d\end{pmatrix*}\in\SL_2(\Z),
$$
and which is also meromorphic at infinity, that is, having a Laurent (Fourier) expansion
$$f(\tau)=\sum_{n\gg -\infty}a_{f}(n)q^n,\quad \text{where }q=e^{2\pi i\tau}.$$
If $f$ is holomorphic on $\Hcal$ but meromorphic at infinity, we say that $f$ is a \emph{weakly holomorphic modular form}. If $f$ is further holomorphic at infinity, i.e. $a_{f}(n)=0$ for all $n<0$, we say that  $f$ is a \emph{holomorphic modular form}. Note that a meromorphic modular form usually has poles on $\Hcal$ and has exponential growth at infinity.

We denote by $\Mcal_k^{\mero}$ (resp. $\Mcal_k^{!}$, $\Mcal_{k}$) the $\C$-vector space of meromorphic (resp. weakly holomorphic, holomorphic) modular forms of weight $k$,  we denote also by
$$
\Mcal^{*}=\bigoplus_{k\in 2\Z}\Mcal_k^{*},\quad *\in\lbrace\text{\mero},!,-\rbrace
$$
the graded $\C$-algebra of meromorphic (resp. weakly holomorphic, holomorphic) modular forms.

As usual, we define for integer $k\geq 2$ the Eisenstein series
$$E_{k}(\tau)=1-\frac{2k}{B_k}\sum_{n=1}^{\infty}\sigma_{k-1}(n)q^n,$$
where $B_k$ is the $k$-th Bernoulli number and $$\sigma_{k-1}(n)=\sum_{d|n}d^{k-1}.$$ 
The Eisenstein series $E_{k}$ are holomorphic modular forms of weight $k$ for $k\geq 4$. In particular, the Eisenstein series $E_{4}$ and $E_{6}$ are algebraically independent and generate the whole graded ring of meromorphic modular forms. To be specific, as graded $\C$-algebras, one has

\begin{align*}
 \Mcal&=\C[E_{4},E_{6}]=\bigoplus_{k\in\N}\bigoplus_{\substack{(a,b)\in\N^2\\4a+6b=k}}\C E_{4}^{a}E_{6}^{b},\\
\Mcal^{!}&=\C[\Delta^{-1},E_{4},E_{6}]=\bigoplus_{k\in\Z}\bigoplus_{\substack{(a,b,n)\in\N^3\\4a+6b-12n=k}}\C\frac{E_{4}^{a}E_{6}^{b}}{\Delta^{n}}
\end{align*}
and
$$ \Mcal^{\mero}=\C[E_{4},E_{6}]_{((0))},$$
where $\Delta=\frac{1}{1728}(E_{4}^3-E_{6}^2)$ is the unique normalized cusp form of weight $12$ and $((0))$ denotes the homogeneous localization at the prime ideal $(0)$.

However in the case $k=2$, the Eisenstein series $E_2$ is no longer modular. In fact, it verifies the following transformation rule
\begin{equation}\label{eq:E2}
 E_2|_{2}\gamma(\tau)=E_2(\tau)+\frac{6}{\pi i}\left(\frac{c}{c\tau+d}\right).   
\end{equation}
for any $\gamma=\begin{pmatrix*}a&b\\c&d\end{pmatrix*}\in\SL_2(\Z)$. In general, we define
\begin{defn}
A \emph{meromorphic quasi-modular form} of weight $k\in\Z$ is a meromorphic function $f$ on $\Hcal$ with a collection of \emph{component functions} $f_0,f_1,\cdots,f_p$ over $\Hcal$, such that
\begin{enumerate}[label=(\roman*)]
    \item each $f_i$ is meromorphic on $\Hcal$ and is also meromorphic at infinity,
    \item the function $f$ verifies the transformation rule
    \begin{equation}\label{eq:slashf}
        (f|_k\gamma)(\tau)=\sum_{r=0}^p f_r(\tau)\left(\frac{c}{c\tau+d}\right)^r,\quad\text{for any } \gamma\in \SL_2(\mbz).
    \end{equation}
\end{enumerate}
If $f_p\neq 0$, the number $p$ is called the \emph{depth} of $f$. We write also $f_r=Q_r(f)$ for the $r$-th component of a meromorphic quasi-modular form $f$.
\end{defn}
\begin{rmk}
It can be seen that the Eisenstein series $E_2$ is of weight $2$ and depth $1$. A meromorphic quasi-modular form of depth $0$ is nothing but a meromorphic modular form.
\end{rmk}

We will denote by $\qm_k^{\mero,\,p}$ (resp. $\qm_k^{!,\,p}$, $\qm_k^{p}$) for the set of meromorphic (resp. weakly holomorphic, holomorphic) quasi-modular forms of weight $k$ and depth $p$ and $\qm_k^{\mero,\,\leq p}$ (resp. $\qm_k^{!,\,\leq p}$, $\qm_k^{\leq p}$) the $\mbc$-vector space of meromorphic (resp. weakly holomorphic, holomorphic) quasi-modular forms of weight $k$ and at most depth $p$. 
Analogously, write
$$\qm^{*}=\bigcup_{p\in\N, k\in\Z}\qm_k^{*,\,p},\quad *\in\lbrace\text{\mero},\,!\, ,-\rbrace,$$
for the $\C$-algebra of meromorphic (resp. weakly holomorphic, holomorphic) quasi-modular forms.

Similar to holomorphic quasi-modular forms, each component function $f_j$ of a meromorphic quasi-modular form $f$ is again a meromorphic quasi-modular form.

\begin{lem}\label{lem:qjf}
Let $f\in \qm_k^{\mero,\,p}$ be a meromorphic quasi-modular form of weight $k$ and depth $p$ with component functions $f_0,\cdots,f_p$. Then for any $0\leq j\leq p$, the function $f_j$ is meromorphic quasi-modular form of weight $k-2j$ and depth $p-j$. More precisely, for any $\gamma\in\SL_2(\mbz)$,
\[(f_j|_{k-2j}\gamma)(\tau)=\sum_{r=0}^{p-j}\binom{j+r}{r}f_{j+r}(\tau)\left(\frac{c}{c\tau+d}\right)^r.\]
In particular, we have that $f_0=f$ and $f_p$ is indeed a meromorphic modular form of weight $k-2p$.
\end{lem}
\begin{proof}
The details of the proof is omitted since it is completely identical with the proof for the holomorphic case given by \cite{Zagier08} (see also \cite[Thm 5.1.22]{CS179}).
\end{proof}
\begin{rmk}
Since there are no holomorphic quasi-modular forms of negative weights, we know a holomorphic quasi-modular form always has weight $k\geq 0$ and depth $p\leq k/2$. However, a meromorphic quasi-modular form can has arbitrary weight $k$ and depth $p$. This may be the first glimpse of how a meromorphic quasi-modular form differs from a holomorphic quasi-modular form.
\end{rmk}
Meromorphic quasi-modular forms, especially those with nonpositive weights, enjoy certain same properties like Lemma~\ref{lem:qjf} compared with holomorphic modular quasi-forms. For the reader's convenience, we will sketch the proofs and reassure the reader that some arguments of holomorphic quasi-modular forms still work here.

\section{Differential operators}\label{sec:3}

This section gives an introduction to the differential operators we will be using. 

Let $k$ and $p\geq 0$ be integers. Let $f$ be a meromorphic function on the upper-half plane. For later use, we recall that the \emph{Maass--Shimura derivative} of $f$ is 
$$
\delta_k f=Df-k\,Y\!f
$$
where $D=\frac{1}{2\pi i}\frac{d}{d\tau}$ and $Y=-\frac{1}{4\pi y}$. We recall also the \emph{Serre derivative} of $f$
$$\vartheta_{k,p}f=Df-\frac{k-p}{12}\,E_{2}f.$$

We will occasionally abbreviate by $\delta$ and $\vartheta$ in an abuse of notations. We will often use the following identities found by Ramanujan.
\[DE_2=\frac{1}{12}(E_2^2-E_4), \quad DE_4=\frac{1}{3}(E_2E_4-E_6), \quad DE_6=\frac{1}{2}(E_2E_6-E_4^2).\]

\begin{lem}\label{lem:deltadfnf}
Let $k$ be an integer and $f$ be a meromorphic function on the upper-half plane. For any $\gamma\in \SL_2(\mbr)$, we have 
\begin{align*}
    (Df)|_{k+2}\gamma&=D(f|_k\gamma)+\frac{k}{2\pi i}\frac{c}{c\tau+d}f|_k\gamma\\
    (\delta_{k} f)|_{k+2}\gamma&=\delta_{k}(f|_k\gamma).
\end{align*}
Moreover, we have the following explicit formula for $n$ times Maass--Shimura derivative.
\begin{equation}\label{eq:deltan}
    \delta_{k}^nf=\sum_{j=0}^n\binom{n}{j}(k+j)_{n-j}Y^{n-j}D^jf,
\end{equation}
where $(a)_n=a(a+1)\cdots(a+n-1)$ and for $n>0$, we set $\delta_{k}^{n}=\delta_{k+2n-2}\circ\delta_{k+2n-4}\circ\dots\circ\delta_{k}$ and $\delta_{k}^{0}$ to be the identity operator.
\end{lem}
\begin{proof}
The first and second identities are immediate calculations, so we just check the last one. 
We prove it by induction on $n$. When $n=1$, the identity $\eqref{eq:deltan}$ is just the definition. Applying Maass--Shimura derivatives on $\eqref{eq:deltan}$, we find that $\delta^{n+1}f$ equals
\begin{align*}
D(\delta^{n}f)&+(k+2n-2)Y\delta^{n}f\\
 =&\sum_{j=0}^{n+1}\bigg(\binom{n}{j-1}(k+j-1)_{n-j+1}Y^{n+1-j}D^jf\\
 &+\binom{n}{j}(k+j)_{n-j}(n+k+j)Y^{n+1-j}D^jf\bigg)\\
        =&\sum_{j=0}^{n+1}\binom{n+1}{j}(k+j)_{n+1-j}Y^{n+1-j}D^jf.
\qedhere
\end{align*}
\end{proof}
In particular, when $k\leq 0$ and $n=1-k$ the coefficients in $\eqref{eq:deltan}$ vanish for $j\neq 0$, in that case we obtain the Bol's identity (cf.~\cite{LZ01})
\begin{equation}
  \delta_{k}^{1-k}f=D^{1-k}f\quad\text{for any } f\in\mcm_{k}^{\mero}.
\end{equation}
This indicates that $D^{1-k}$ is in fact an $\SL_2(\Z)$-invariant differential operator on $\mcm_{k}^{\mero}$. This features, as we will see later, will significantly change the behaviors of meromorphic quasi-modular forms.

\begin{lem}\label{lem:dfqp}
Given a meromorphic quasi-modular form $f\in \qm_k^{\mero,\,p}$ with the component functions $f_0,f_1,\cdots,f_p$, we have
\[(Df)|_k\gamma(\tau)=\sum_{r=0}^{p+1}\left(D(f_r)+\frac{k-r+1}{2\pi i}f_{r-1}\right)\left(\frac{c}{c\tau+d}\right)^r,\]
where by convention $f_{-1}=f_{p+1}=0$.
\end{lem}
\begin{proof}
Since $D(\frac{c}{c\tau+d})=-\frac{1}{2\pi i}(\frac{c}{c\tau+d})^2$, by applying the differential operator $D$ to both sides of $\eqref{eq:slashf}$, we get 
\[D(f|_k\gamma)=\sum_{r=0}^pD(f_r)\left(\frac{c}{c\tau+d}\right)^{r}-\frac{r}{2\pi i}f_r(\tau)\left(\frac{c}{c\tau+d}\right)^{r+1}.\]
We complete the proof by applying Lemma \ref{lem:deltadfnf}. 
\end{proof}

The following lemma implies that the differential operator $D$ usually increases the depth of a meromorphic quasi-modular form by $1$ (note that in holomorphic case it always does).  
\begin{prop}\label{prop:depthd}
The space of meromorphic quasi-modular forms $\qm^{\mero}$ is stable under the derivation $D$. It acts on $\qm^{\mero}$ by increasing the weight by $2$ and increasing the depth by at most $1$
$$D:\qm_{k}^{\mero,\,\leq p}\rightarrow\qm_{k+2}^{\mero,\,\leq p+1}.$$
More precisely, for a meromorphic modular form $f\in \qm_k^{\mero,\,p}$, we have  $Df\in \qm_{k+2}^{\mero,\,p+1}$ if $k\neq p$ and $Df\in \qm_{k+2}^{\mero,\,\leq p}$ if $k=p$.
\end{prop}
\begin{proof}
This follows immediately from Lemma~\ref{lem:dfqp}. Note that one has
$$Q_{p+1}(Df)=\f{k-p}{2\pi i}Q_{p}(f).$$
So $Df$ will have exactly depth $p+1$ unless $k=p$.
\end{proof}

The Serre derivative, however, usually preserves the depth of a quasi-modular form. We now state an analogous result.
\begin{prop}
The space of meromorphic quasi-modular forms $\qm^{\mero}$ is stable under the Serre derivative $\vartheta$. It acts on $\qm^{\mero}$ by increasing the weight by $2$,
$$\vartheta_{k,p}:\qm_{k}^{\mero,\,\leq p}\rightarrow\qm_{k+2}^{\mero,\,\leq p}.$$
In particular, if $f$ is a meromorphic modular form of weight $k$, then $\vartheta_{k,p} f$ is a meromorphic modular form of weight $k+2$.
\end{prop}
\begin{proof}
Using the modular transformation of $E_2$ in equation~\eqref{eq:E2} and modular transformation of $Df$ in Lemma~\ref{lem:dfqp},
we get
$$Q_{p+1}(Df)=\frac{k-p}{2\pi i}Q_{p}(f),\quad Q_{p+1}(E_2 f)=\frac{6}{\pi i}Q_{p}(f).$$
Hence we always have $Q_{p+1}(\vartheta_{k,p} f)=0$. Thus $\vartheta_{k,p} f$ has at most depth $p$.
\end{proof}

\begin{lem}\label{lem:depthmodular}
Let $f\in \mcm_k^{\mero}$ be a meromorphic modular form of weight $k$. Then

\begin{enumerate}[label=(\roman*)]
    \item If $k>0$, we have $D^pf\in \qm_{k+2p}^{\mero,\,p}$,
    \item If $k\leq 0$ and $p\leq -k$, then we have $D^pf\in \qm_{k+2p}^{\mero,\,p}$,
   \item If $k\leq 0$ and $p\geq 1-k$, then we have $D^pf\in \qm_{k+2p}^{\mero,\,p+k-1}$.
\end{enumerate}
\end{lem}
\begin{proof}
Applying iteratively Lemma \ref{lem:dfqp}, we obtain
\begin{equation}\label{eq:qpdpf}
Q_{p}(D^{p}f)=\frac{p!}{(2\pi i)^p}\binom{k+p-1}{p}\,f,
\end{equation}
which is always nonvanishing when $k>0$ or $p+k\leq 0$. This proves the assertions $(\rom{1})$ and $(\rom{2})$.

For the third one, keep in mind that $D^{1-k}f$ is a meromorphic modular form of weight $2-k$ by Bol's identity. Therefore by the assertion $(\rom{1})$, we have
\[
D^pf=D^{p+k-1}(D^{1-k}f)\in \qm_{k+2p}^{\mero,\,p+k-1}.
\qedhere
\]
\end{proof}

\begin{lem}
Let $p$ be a nonnegative integer. Then the sequence
\begin{equation*}
0 \longrightarrow \delta_{p=0}\,\C[\Delta^{\pm}] \longrightarrow \qm^{\mero,\,\leq p} \stackrel{\vartheta}{\longrightarrow} \qm^{\mero,\,\leq p}
\end{equation*}
of graded  $\C$-vector space is left exact where $\delta_{p=0}=1$ when $p=0$ and equals to $0$ when $p\neq 0$.
\end{lem}
\begin{proof}
Let $f\in\qm_{k}^{\mero,\,p}$ be a meromorphic quasi-modular form in the kernel of the Serre derivative $\vartheta$. As $E_2= D\Delta/\Delta$, we have the identity of logarithmic derivative  $12Df/f=(k-p) D\Delta/\Delta$.  Therefore $f$ is a power of $\Delta$ with $k\in 12\Z$ and $p=0$.
\end{proof}

\section{Structure of meromorphic quasi-modular forms}\label{sec:4}
In this section, we study the structure of the graded ring of meromorphic quasi-modular forms.

Throughout this section, we set $*\in\lbrace\mero,!\rbrace$. The $\C$-algebra of meromorphic quasi-modular forms is graded by the weight $k$ and filtered by the depth $p$
$$\qm^{*}=\bigoplus_{k\in\Z}\bigcup_{p\in\N}\qm_{k}^{*,\,\leq p}.$$
\begin{lem}\label{lem:exact}
Let $k$ and $p$ be integers with $p\geq 0$. Then we have the following split exact sequence
\begin{equation*}
0 \longrightarrow \qm_k^{*,\,\leq p-1} \longrightarrow \qm_k^{*,\,\leq p} \stackrel{Q_{p}}{\longrightarrow} \mcm_{k-2p}^{*} \longrightarrow 0.
\end{equation*}
\end{lem}
\begin{proof}
Let $f\in\qm_k^{*,\,\leq p}$ be a meromorphic quasi-modular form. We recall that by Lemma \ref{lem:qjf}, the last component $Q_p(f)\in\mcm_{k-2p}^{*}$ is in fact a meromorphic modular form of weight $k-2p$. Thus the above sequence is exact. From the modular transformation~\eqref{eq:E2} of $E_2$, we know that $Q_p(gE_2^{p})=\left(6/\pi i\right)^{p}g$ for any $g\in \mcm_{k-2p}^{*}$.
Thus the map $g\mapsto (\frac{2\pi i}{12})^{p}gE_2^{p}$ is a section of the map $Q_p$, so the above exact sequence is also split.
\end{proof}

\begin{thm}\label{thm:fE2}
The graded $\C$-algebra of meromorphic quasi-modular forms is generated by meromorphic modular forms and $E_2$
$$\qm^{*}=\mcm^{*}[E_2]=\bigoplus_{p\geq 0}\mcm^{*}E_2^p, \quad*\in\lbrace\mero,\,!\rbrace,$$
where the depth of a meromorphic quasi-modular forms is exactly the degree of $E_2$ within it.
\end{thm}
\begin{proof}
Induction on the depth of $f\in\qm^{*}$. The statement is straightforward when the depth is $0$. As explained in Lemma~\ref{lem:exact}, we see that the form $f-(\frac{2\pi i}{12})^pQ_p(f)E_2^p$ has depth $\leq p-1$.  Therefore by induction, for any $f\in \qm_k^{*,\,p}$ there exist meromorphic modular forms $g_i$ of weight $k-2i$ such that $f=\sum_{i=0}^pg_iE_2^i$.
\end{proof}

Let $\gr_{p}\qm^{*}$ be the associated graded $\C$-algebra with respect to the depth $p$. Then $\gr_{p}\qm^{*}$ is a bigraded ring
$$\gr_{p}\qm^{*}\simeq\bigoplus_{k\in\Z}\bigoplus_{p\in\N}\mcm_{k-2p}^{*}E_2^{p}.$$
Then the induced derivative $\gr_{p} D$ on $\gr_{p}\qm^{*}$ is homogeneous, increasing the weight by $2$ and depth by $1$.

\begin{prop}\label{prop:bijd}
We have the following left exact sequence of bigraded $\C$-vector spaces
\begin{equation*}
0 \longrightarrow \bigoplus_{p\in\N} \mcm_{-p}^{*}\,E_2^{p} \longrightarrow \gr_{p}\qm^{*} \xrightarrow{\gr_{p}D} \gr_{p}\qm^{*} .
\end{equation*}
In fact, the induced derivative $\gr_p D$
\[\gr_p D:\qm_k^{*,\,\leq p}/\qm_k^{*,\,\leq p-1}\to \qm_{k+2}^{*,\,\leq p+1}/\qm_{k+2}^{*,\,\leq p},\]
is a bijection when $p\neq k$ and is a zero map if $p=k$.
\end{prop}
\begin{proof}
It follows from Lemma \ref{lem:depthmodular} that the map $\gr_p D$ is injective when $p\neq k$ and is a zero map when $p=k$. We next show the surjectivity of $\gr_p D$ for $p\neq k$. Using the fundamental identity $DE_2=\frac{1}{12}(E_2^2-E_4)$, we observe that for a meromorphic modular form $g$ of weight $k-2p$, 
\begin{equation}\label{eq:dfk2p}
    \begin{split}
        &D(gE_2^p)=Dg\,E_2^{p}+\frac{p}{12}g(E_2^{p+1}-E_2^{p-1}E_4)\\
    =&\vartheta g\,E_2^{p}+\frac{k-p}{12}gE_2^{p+1}-\frac{p}{12}gE_4E_2^{p-1}.
    \end{split}
\end{equation}
Here we recall that $\vartheta g=Dg-\frac{k-2p}{12}E_2g$ is the Serre derivative of $g$, which is modular of weight $k-2p+2$. Hence
\[D(gE_2^p)\in \frac{k-p}{12}gE_2^{p+1}+\qm_{k+2}^{\mero,\,\leq p}.\]
Since $k\neq p$, the map $\gr_p D$ is surjective.
\end{proof}
Likewise, we have an induced homogeneous Serre derivative $\gr_{p}\vartheta$ on $\gr_{p}\qm^{*}$, which increases the weight by $2$ and preserves the depth.
\begin{prop}
We have the following left exact sequence of bigraded $\C$-vector spaces
\begin{equation*}
0 \longrightarrow \C[\Delta^{\pm},\,E_2] \longrightarrow \gr_{p}\qm^{*} \xrightarrow{\gr_{p}\vartheta} \gr_{p}\qm^{*} .
\end{equation*}
\end{prop}
\begin{proof}
A direct computation
\begin{align*}
\vartheta(gE_2^{p})&=p\vartheta(E_2)\,E_2^{p-1}g+ \vartheta(g)E_2^{p}\\
&=-\frac{p}{12}g E_4 E_2^{p-1}+\vartheta g\,E_2^{p}
\end{align*}
yields that $\vartheta(gE_2^{p})\in\vartheta g\,E_2^{p}+\qm_{k+2}^{\mero,\,\leq p-1}$. Therefore the kernel of $\gr_{p}\vartheta$ is the graded subalgebra generated by powers of $\Delta$ and $E_2$.
\end{proof}


It is well-known that a holomorphic quasi-modular form is always a linear combination of iterated derivatives of holomorphic modular forms and $E_2$. The theorem below shows this fails for meromorphic quasi-modular forms. This is one of the major differences between the holomorphic case and meromorphic case.

\begin{thm}\label{prop:fdinfty}
Let $f\in\qm_k^{\mero,\,p}$ be a meromorphic quasi-modular form with depth $k/2\leq p<k$. Then $f$ is never a linear combination of iterated derivatives of meromorphic modular forms.
\end{thm}
\begin{proof}
Suppose that $f$ admits a decomposition $f=\sum_{l=0}^{l_0} D^l F_{k-2l}$, where $F_{k-2l}$ is a meromorphic modular form of weight $k-2l$. Lemma \ref{lem:depthmodular} shows that  when $l_0<k/2$ or $l_0\geq k$, the depth of $f$ is exactly $l_0$. So we are left to consider only $k/2\leq l_0<k$. If $k/2\leq l\leq l_0$, the depth of $D^l F_{k-2l}$ is $k-l-1$, which is strictly smaller than $k/2$, and if $l<k/2$ the depth of $D^l F_{k-2l}$ is $l$, also strictly smaller than $k/2$. This implies that the depth of $f$ is always $<k/2$, which leads to a contradiction.
\end{proof}

Finally, we give the proof of the decomposition of the space of meromorphic quasi-modular forms.

\begin{proof}[Proof of Theorem~\ref{thm:decomposition}]
Lemma~\ref{lem:depthmodular} indicates that each component in the first and last parts should have different depth $l\leq k/2-1$ and $l\geq k$ respectively. In the middle part, applying repeatedly Proposition~\ref{prop:depthd} we find that each component has depth $k-l-1$ for $0\leq l\leq k/2-1$. So the above sum runs through all depths and must be a direct sum of $\C$-vector spaces.  It remains to show every meromorphic quasi-modular form $f\in \qm_k^{*,\,p}$ has such decomposition. We divided the proof into three parts.
\proofpart{1}{$p\leq k/2-1$}
When $p=0$, the result is direct. For $0<p\leq k/2-1$, on account of the computation in~\eqref{eq:qpdpf} (or using Proposition~\ref{prop:bijd}), the $p$-th component function of 
\[f-\frac{(2\pi i)^p}{p!\binom{k-p-1}{p}}D^pQ_p(f)\]
is zero, so we complete the proof by induction on $p$.

\proofpart{2}{$p\geq k$}
In this case, since $\binom{k-p-1}{p}\neq 0$, the induction argument still works unless the depth goes to less than $k$. This implies that we can find $F_{k-2l}\in \mcm_{k-2l}^*$ where $l=k,\dots,p$ and $F\in \qm_k^{*,\,\leq k-1}$ such that 
\[f=D^pF_{k-2p}+\cdots+D^kF_{-k}+F.\]
So we reduce the case $p\geq k$ to the case $p<k$.

\proofpart{3}{$k/2 \leq p<k$}

We claim that for any such $f$ there exists a meromorphic quasi-modular form $h\in\bigoplus_{l=0}^{\frac{k}{2}-1}D^l\qm_{k-2l}^{*,\,k-2l-1}$ so that $f-h\in \qm_k^{*,\,\leq k/2-1}$. 
Then the proof will be converted to the first part. We will prove the claim by induction on $k+p$.

Starting from $k+p=3$, the only possible case is $k=2$, $p=1$. In this case $f\in \qm_2^{*,\,1}$, the result is direct since it corresponds to $l=0$. We assume that $k+p>3$. If $p=k-1$, then the result is also direct since $f\in \qm_k^{*,\,k-1}$.

Then we can assume that $p<k-1$. By Theorem~\ref{thm:fE2}, there exists a meromorphic modular form $F_{k-2p}$ of weight $k-2p$ such that $f=g+F_{k-2p}\,E_2^p$ for some $g\in \qm_k^{*,\,\leq p-1}$. The same calculation as in equation~\eqref{eq:dfk2p} yields that $F_{k-2p}\,E_2^p$ can be represented as the linear combination
$$
\frac{k-p-1}{12}F_{k-2p}\,E_2^p=D(F_{k-2p}\,E_2^{p-1})-\vartheta (F_{k-2p})E_2^{p-1}+\frac{p-1}{12}F_{k-2p}\,E_4E_2^{p-2}.
$$
On the right-hand side, the first term comes from the derivative of $F_{k-2p}\,E_2^{p-1}$, which has exactly weight $k-2$ and depth $p-1$. So by induction assumption, we can find a weight $k-2$ meromorphic quasi-modular form $G_{k-2}$ such that
$$F_{k-2p}\,E_2^{p-1}-G_{k-2}\in \qm_{k-2}^{*,\,\leq \frac{k-2}{2}-1}\hspace{0.5em}\text{with   }\,G_{k-2}\in\!\! \bigoplus_{l=0}^{\frac{k-2}{2}-1}D^l\qm_{k-2l-2}^{*,\,k-2l-3}.$$ 
Then by applying the operator $D$, we get 
\[D(F_{k-2p}\,E_2^{p-1})-D(G_{k-2})\in \qm_{k}^{*,\,\leq \frac{k}{2}-1},\]
where 
\[D(G_{k-2})\in \bigoplus_{l=1}^{\frac{k}{2}-1}D^l\qm_{k-2l}^{*,\,k-2l-1}.\]
Besides, the second term $\vartheta (F_{k-2p})\,E_2^{p-1}$ and the last term $F_{k-2p}\,E_4E_2^{p-2}$ on the right-hand side and and the function $g$ all have weight $k$ and depth $\leq p-1$. So we find that $$f-\frac{12}{k-p-1}D(G_{k-2})\in \qm_{k}^{*,\,\leq p-1}.$$
By induction assumption, we know there exists $H\in\bigoplus_{l=0}^{\frac{k}{2}-1}D^l\qm_{k-2l}^{*,\,k-2l-1}$ so that
$$f-\frac{12}{k-p-1}D(G_{k-2})-H\in \qm_{k}^{*,\,\leq \frac{k}{2}-1},$$
which proves the previous claim.
\end{proof}

\begin{rmks}
\begin{enumerate}[label=(\roman*)]
    \item Theorem~\ref{thm:decomposition} shows that every meromorphic quasi-modular form can be written uniquely (up to Bol's identity) as a linear combination of iterated derivatives of meromorphic modular forms and iterated derivatives of quasi-modular forms of weight $l$ and depth $l-1$.
    \item When $k\leq 0$, we get only the first part, thus every meromorphic quasi-modular form of nonpositive weight is just a linear combination of iterated derivatives of meromorphic modular forms of nonpositive weights.
    \item For holomorphic quasi-modular form this reduces to the well-known (see Zagier~\cite[Prop. 20]{Zagier08})
    $$
\qm_k=\bigg(\bigoplus_{l=0}^{\frac{k}{2}-2}D^l\mcm_{k-2l}\bigg)\bigoplus D^{\frac{k}{2}-1}\qm_{2}^{1},
$$
where the depth $p=k/2$ comes from the iterated derivatives of the Eisenstein series $E_2$ which generates $\qm_{2}^{1}$.
    \item According to the work of Paşol and Zudilin~\cite{PZ20}, it is reasonable to conjecture that all magnetic meromorphic quasi-modular forms come from iterated derivatives with $l>0$ of (quasi-)modular forms with Fourier expansion in $\Q\otimes_{\Z}\Z[[q]]$.
\end{enumerate}
\end{rmks}

\section{Rankin--Cohen brackets of meromorphic modular forms}\label{sec:5}
In this section, we introduce the Rankin--Cohen brackets of meromorphic modular forms. The Rankin--Cohen brackets of holomorphic modular forms has been studied in many literature. The reader can find details in \cite{CS179}.

We first introduce the Cohen--Kuznetsov series associated with a meromorphic modular form. For holomorphic modular forms these series were originally introduced by Cohen \cite{Coh75} and Kuznetsov \cite{Kuz75}. When $k$ is a positive integer, our series is the same as Cohen and Kuznetsov. When $k$ is a negative integer, we will define a minus series and a plus series.
\begin{defn}
Let $f$ be a meromorphic modular form of weight $k\in\mbz$. We define its \emph{Cohen--Kuznetsov series} by 
\begin{align*}
    CK_D^-(f;\tau,T)&=\sum_{n=0}^{-k}(-1)^{n+k}\frac{(-k-n)!}{n!}D^nf\,T^n,\\
    CK_D^+(f;\tau,T)&=\sum_{n\geq 1-k}\frac{D^nf}{n!(n+k-1)!}T^n.
\end{align*}
Similarly, we can define $CK_{\delta}^{\pm}(f;\tau,T)$ by replacing $D$ by $\delta$. When $k>0$, by convention, $CK^-=0$, and $CK^+$ starts from the term $n=0$. Moreover, we define the slash operator on $CK^{\pm}$ by 
\[(CK^{\pm}|_k\gamma)(f;\tau,T)=(c\tau+d)^{-k}CK^{\pm}\left(f;\frac{a\tau+b}{c\tau+d},\frac{T}{(c\tau+d)^2}\right)\]
where $\gamma\in\SL_2(\mbr)$.
\end{defn}

We first give some basic properties of Cohen--Kuznetsov series. When $f$ has positive weight, similar results can also be found in \cite{CS179}. 
\begin{prop}\label{prop:ckplus}
Let $f$ be a meromorphic function of weight $k$. Suppose $k$ is a nonpositive integer. Then
\begin{enumerate}[label=(\roman*)]
    \item We have 
    \[CK_{*}^+(f;\tau,T)=T^{1-k}CK_{*}^+(D^{1-k}f;\tau,T),\]
    where $*$ denotes the operator $D$ or $\delta$.
    \item The functions $CK_{\delta}^+$ and $CK_D^+$ are linked by
    \[CK_{\delta}^+(f;\tau,T)=e^{TY}CK_D^+(f;\tau,T).\]
    \item The series $CK_{\delta}^+$ and $CK_D^+$ commutes with the slash operator up to a factor. Namely, for any $\gamma\in \SL_2(\mbr)$, we have
    \begin{align*}
    (CK_{\delta}^+)|_k\gamma(f;\tau,T)&=CK_{\delta}^+(f|_k\gamma;\tau,T),\\
    (CK_{D}^+)|_k\gamma(f;\tau,T)&=e^{\frac{T}{2\pi i}\frac{ c}{c\tau+d}}CK_{D}^+(f|_k\gamma;\tau,T). 
    \end{align*}

\end{enumerate}
\end{prop}
\begin{proof}
$(\rom{1})$. Since $D^{1-k}f$ is a modular form of weight $2-k>0$, one has 
\[T^{1-k}CK_D^+(D^{1-k}f;\tau,T)=\sum_{n\geq 0}\frac{D^n(D^{1-k}f)}{n!(n+1-k)!}T^{n+1-k}=CK_D^+(f;\tau,T).\]
The proof for $CK_{\delta}^+$ is similar.

$(\rom{2})$. We note that if $f$ is a meromorphic modular form of positive weight $k$, then by Lemma \ref{lem:deltadfnf}, we have 
\begin{align*}
    CK_{\delta}^+(f;\tau,T)=&\sum_{n\geq 0}\sum_{j=0}^n\frac{\binom{n}{j}(k+j)_{n-j}Y^{n-j}D^jf}{n!(n+k-1)!}T^n\\
    =&\sum_{j,\,l\geq 0}\frac{Y^lD^jf}{(k+j-1)!l!j!}T^{l+j}\\
    =&e^{YT}CK_D^+(f;\tau,T).
\end{align*}
So when the weight of $f$ is nonpositive, we can apply the result above to $D^{1-k}f$ and get
\begin{align*}
CK_{\delta}^+&(f;\tau,T)=T^{1-k}CK_{\delta}^+(D^{1-k}f;\tau,T)\\
&=T^{1-k}e^{TY}CK_{\delta}^+(D^{1-k}f;\tau,T)=e^{TY}CK_D^+(f;\tau,T).   
\end{align*}

$(\rom{3})$. Similar to the proof above, we start from the positive weight case. By Lemma \ref{lem:deltadfnf}, we have 
\[(c\tau+d)^{-2n}(\delta^nf)\left(\frac{a\tau+b}{c\tau+d}\right)=\delta^n(f|_k\gamma).\]
Then 
\begin{align*}
    &(CK_{\delta}^+)|_k\gamma(f;\tau,T)=\sum_{n\geq 0}\frac{(\delta^nf)(\frac{a\tau+b}{c\tau+d})}{n!(n+k-1)!}\frac{T^n}{(c\tau+d)^{2n+k}}\\
    =&\sum_{n\geq 0}\frac{\delta^nf(\tau)}{n!(n+k-1)!}\frac{T^n}{(c\tau+d)^k}=CK_{\delta}^+(f|_k\gamma;\tau,T).
\end{align*}
Finally, we can complete the proof by applying $(\rom{1})$
\begin{align*}
    (CK_{\delta}^+)&|_k\gamma(f;\tau,T)=(c\tau+d)^{-k}CK_{\delta}^+\left(f;\frac{a\tau+b}{c\tau+d},\frac{T}{(c\tau+d)^2}\right)\\
    =&(c\tau+d)^{k-2}CK_{\delta}^+\left(D^{1-k}f;\frac{a\tau+b}{c\tau+d},\frac{T}{(c\tau+d)^2}\right)T^{1-k}\\
    =&T^{1-k}CK_{\delta}^+((D^{1-k}f)|_{2-k}\gamma;\tau,T)=CK_{\delta}^+(f|_k\gamma;\tau,T).
\end{align*}
As for the series $CK_D^+$, we just need to apply $(\rom{2})$. 
\end{proof}

Then we focus on the minus Cohen--Kuznetsov series.

\begin{prop}\label{prop:ckminus}
Let $f$ be a meromorphic modular form of weight $k$. Suppose $k$ is a nonpositive integer. Then for the minus part of Cohen--Kuznetsov series, we have the following relations
\begin{enumerate}[label=(\roman*)]
    \item The functions $CK_{\delta}^-$ and $CK_D^-$ are linked by
    $$CK_{\delta}^-(f;\tau,T)=e^{TY}CK_D^-(f;\tau,T)+O(T^{1-k}).$$
    \item The function $CK_{\delta}^-$ also commutes with the slash operator, i.e. for any $\gamma\in \SL_2(\mbr)$, we have  
    \[(CK_{\delta}^-)|_k\gamma(f;\tau,T)=CK_{\delta}^-(f|_k\gamma;\tau,T).\]
    \item The function $CK_D^-$ has the transformation 
    \[(CK_{D}^-)|_k\gamma(f;\tau,T)=e^{\frac{T}{2\pi i}\frac{ c}{c\tau+d}}CK_{D}^-(f|_k\gamma;\tau,T)+O(T^{1-k}).\]
\end{enumerate}
\end{prop}
\begin{proof}
$(\rom{1})$. The proof is similar to $CK_{\delta}^+$. By definition, we have
\begin{align*}
    CK_{\delta}^-(f;\tau,T)=\sum_{n=0}^{-k}\sum_{j=0}^n\binom{n}{j}(k+j)_{n-j}\frac{(-1)^{n+k}(-n-k)!}{n!}Y^{n-j}D^jf\,T^n
\end{align*}
By changing $n=j+l$, we get
\begin{align*}
    &\sum_{j=0}^{-k}\frac{(-1)^{k-j}(-k-j)!}{j!}D^jf\,T^j\sum_{l=0}^{-k-j}\frac{Y^l}{l!}T^l\\
    =&\sum_{j=0}^{-k}\frac{(-1)^{k-j}(-k-j)!}{j!}D^jf\,T^j\sum_{l=0}^{\infty}\frac{Y^l}{l!}T^l+O(T^{1-k})\\
    =&e^{TY}CK_D^-(f;\tau,T)+O(T^{1-k}).
\end{align*}
The proof of $(\rom{2})$ and $(\rom{3})$ are similar to $(\rom{3})$ in Proposition~\ref{prop:ckplus}, we omit the details here.
\end{proof}

\begin{defn}
Let $f,g$ be two meromorphic modular forms of weight $k,l$ respectively. Then the \emph{$n$-th Rankin--Cohen bracket} of $f$ and $g$ is defined by 
\[[f,g]_n\coloneqq\sum_{j=0}^n(-1)^j\binom{n+k-1}{j}\binom{n+l-1}{n-j}D^{n-j}fD^jg.\]
\end{defn}

\begin{thm}
Let $f,g$ be two meromorphic modular form of weight $k,l$ respectively. Then the $n$-th Rankin--Cohen bracket $[f,g]_n$ is a meromorphic modular form of weight $k+l+2n$. 
\end{thm}
\begin{proof}
The case of $k,l>0$ is a result of Cohen \cite{CS179}. So we assume that at least one of $k,l$ is nonpositive. By symmetry $[f,g]_n=(-1)^n[g,f]_n$, we may assume that $k\geq l$. There are exactly two possibilities,  either one of $k$ and $l$ is positive or none of $k$ and $l$ is positive. For these two parts, we further separate them into the several cases depending on $n$.
\proofpart{1}{$k>0$, $l\leq 0$} 
\textit{Case 1. $0\leq n\leq -l$}. We consider the product $CK_D^+(f;\tau,T) CK_D^-(g;\tau,-T)$, it is equal to 
$$
CK_D^+(f;\tau,T) CK_D^-(g;\tau,-T)=
\sum_{n=0}^{-l}\sum_{j=0}^{n}(-1)^j A_{j}^{k,l,n} D^{n-j}fD^jg+O(T^{1-l}),
$$
where
\begin{align*}
    A_{j}^{k,l,n}&=(-1)^{l+j}\frac{(-l-j)!}{j!(n-j)!(n+k-j-1)!}\\
    &=(-1)^{n+l}\frac{(-l-n)!}{(k+n-1)!} \binom{k+n-1}{j}\binom{n+l-1}{n-j}.
\end{align*}
Here we use the identity $\binom{-n}{j}=(-1)^j\binom{n+j-1}{j}$ for $n\geq 0$. Thus we have
$$
CK_D^+(f;\tau,T) CK_D^-(g;\tau,-T)=\sum_{n=0}^{-l}\frac{(-1)^{n+l}(-l-n)!}{(k+n-1)!}[f,g]_nT^n+O(T^{1-l}).
$$
On the other hand, by Proposition \ref{prop:ckplus} and \ref{prop:ckminus}, for any $\gamma\in\SL_2(\mbz)$, we have 
\[CK_D^+(f;\tau,T)CK_D^-(g;\tau,-T)|_{k+l}\gamma=CK_D^+(f;\tau,T)CK_D^-(g;\tau,-T)+O(T^{1-l}).\]
This implies that $[f,g]_n$ is a meromorphic modular form of weight $k+l+2n$. 

\textit{Case 2. $n\geq 1-l$}. Then the terms where $j<1-l$ in the Rankin--Cohen bracket vanish since the binomial coefficients become zero. Put $n'=n+l-1$, the Rankin--Cohen  bracket of $f$ and $g$ turns out to be
\begin{equation}\label{eq:dfgn}
\begin{split}
   [f&,g]_{n}=\sum_{j=1-l}^{n}(-1)^{j}\binom{n+k-1}{n-j}\binom{n+l-1}{j}D^{n-j}fD^{j}g\\   &=\frac{\binom{2n+k+l-2}{n}}{\binom{2n+k+l-2}{n'}}\sum_{j=0}^{n'}(-1)^{l+1+j}\binom{n'+k-1}{n'-j}\binom{n'+1-l}{j}D^{n'-j}fD^j(D^{1-l}g)
\end{split}
\end{equation}
Therefore,
$$[f,g]_{n}=(-1)^{l-1}\frac{\binom{2n+k+l-2}{n}}{\binom{2n+k+l-2}{n'}}\,[f,D^{1-l}g]_{n'}.$$
Notice that $D^{1-l}g$ is a meromorphic modular form of positive weight $2-l$, so $[f,g]_n$ is a multiple of $[f,D^{1-l}g]_{n'}$, thus also a meromorphic modular form of weight $k+l+2n$.

\proofpart{2}{$k$, $l\leq 0$}
\textit{Case 1. $0\leq n\leq -k$}. we consider the product $CK_D^-(f;\tau,T)CK_D^-(g;\tau,-T)$. With the same calculation as in Part 1, the product is
\begin{align*}
 CK_D^-(f;&\tau,T)CK_D^-(g;\tau,-T)\\
    &=\sum_{m=0}^{-k}(-1)^{k+l}(-k-m)!(-l-m)![f,g]_nT^m+O(T^{1-k}). 
\end{align*}
So the Rankin--Cohen  bracket $[f,g]_n$ is again a meromorphic modular form of weight $k+l+2n$. 

\textit{Case 2. $1-k\leq n\leq -l$}. Put $n'=n+k-1$, the same calculation as in $\eqref{eq:dfgn}$ shows that 
\begin{align*}
   [f&,g]_{n}=\sum_{j=0}^{n+k-1}(-1)^j\binom{n+k-1}{j}\binom{n+l-1}{n-j}D^{n-j}fD^jg\\   &=\frac{\binom{-l}{n}}{\binom{-l}{n'}}\,\sum_{j=0}^{n'}(-1)^j\binom{n'+1-k}{j}\binom{n'+l-1}{n'-j}D^{n'-j}(D^{1-k}f)D^jg
\end{align*}
Thus
$$[f,g]_{n}=\frac{\binom{-l}{n}}{\binom{-l}{n'}}\,[D^{1-k}f,g]_{n'}$$
is a multiple of $[D^{1-k}f,g]_{n'}$, which reduces to Case 1 in  Part 1, since $D^{1-k}f$ is a meromorphic modular form of positive weight $2-k$.

\textit{Case 3. $1-l\leq n\leq 1-k-l$}. We note that $\binom{n+k-1}{j}$ is non-vanishing if and only if $j\leq n+k-1$ and $\binom{n+l-1}{n-j}$ is non-vanishing if and only if $1-l\leq j$. So when $n\leq 1-k-l$, at least one of these two binomials vanishes, this implies that the Rankin--Cohen bracket
$$
[f,g]_{n}=\sum_{j=0}^{n}(-1)^j\binom{n+k-1}{j}\binom{n+l-1}{n-j}D^{n-j}fD^jg\equiv 0
$$
is always vanishing in this case.

\textit{Case 4. $n\geq 2-k-l$}. Similar to equation $\eqref{eq:dfgn}$, letting $n''=n+k+l-2$ we find that 
\begin{align*}
   [&f,g]_{n}=\sum_{j=1-l}^{n+k-1}(-1)^j\binom{n+k-1}{j}\binom{n+l-1}{n-j}D^{n-j}fD^jg\\   &=\sum_{j=0}^{n''}(-1)^{j+l-1}\binom{n''+1-k}{j}\binom{n''+1-l}{n''-j}D^{n''-j}(D^{1-k}f)D^j(D^{1-l}g)
\end{align*} 
Thus $$[f,g]_n=(-1)^{l-1}[D^{1-k}f,D^{1-l}g]_{n''}$$
is still a meromorphic modular form of weight $k+l+2n$.
\end{proof}

The following theorem by Lanphier \cite{Lan08} and El Gradechi \cite{Gra13}, originally stated for positive weights modular forms, can be also extended to negative weights modular forms.
\begin{thm}
Let $f,g$ be two meromorphic modular forms of weight $k$, $l$ respectively. Let $n$ be a positive integer. Suppose that $k+l\geq 2$ or $k+l+2n\leq 0$. Then we have 
\begin{equation}\label{eq:dfgc}
    D^{n-i}[f,g]_i=\sum_{j=0}^nc_{i,j}^{k,l,n}(D^{n-j}f)(D^jg),
\end{equation}
and
\begin{equation}\label{eq:dfgb}
    (D^{n-i}f)(D^ig)=\sum_{j=0}^nb_{i,j}^{k,l,n}D^{n-j}[f,g]_j,
\end{equation}
where 
\[c_{i,j}^{k,l,n}=\sum_{r=0}^i(-1)^r\binom{n-i}{n-j-r}\binom{k+i-1}{i-r}\binom{l+i-1}{r},\]
and
\[b_{i,j}^{k,l,n}=\frac{\binom{n}{j}\sum_{r=0}^j(-1)^r\binom{j}{r}\binom{k+n-i-1}{n-i-r}\binom{l+i-1}{r+i-j}}{\binom{n}{i}\binom{k+l+n+j-1}{n-j}\binom{k+l+2j-2}{j}}.\]
\end{thm}
\begin{proof}
Proposition 4.6 in \cite{Gra13} shows that for positive weights $k$, $l$ the constants $b_{i,j}^{k,l,n}$ and $c_{i,j}^{k,l,n}$ are mutually inverse
\begin{align}\label{eq:b&c}
  \sum_{r=0}^{n}b_{i,r}^{k,l,n}c_{r,j}^{k,l,n}=\delta_{i,j}
=\sum_{r=0}^{n}c_{i,r}^{k,l,n}b_{r,j}^{k,l,n}.  
\end{align}

We note that for fixed $i$, $j$ and $n$, the coefficients $c_{i,j}^{k,l,n}$ and $b_{i,j}^{k,l,n}$ are actually polynomials in $k,l$. Since $\eqref{eq:b&c}$ holds for all positive integers $k$, $l$, it still holds for all integers $k$, $l$. As long as the denominator of $b_{i,r}^{k,l,n}$ is non-zero, the identities~\eqref{eq:dfgc} and~\eqref{eq:dfgb} remain valid.
\end{proof}

Define now the following two $\C$-subalgebras of $\qm^{*}$ with $*\in\lbrace\mero,!\rbrace$
$$\qm^{*}_{+}=\bigg(\bigoplus_{k>0}\bigoplus_{p=0}^{ \frac{k}{2}-1}\qm_{k}^{*,\,p}\bigg)\bigoplus\C,\quad
\qm^{*}_{-}=\bigoplus_{k\leq 0}\bigoplus_{p\geq 0}\qm_{k}^{*,\,p}.
$$
Then the space $\qm^{*}_{+}$ is generated by positive weight derivatives of positive weight meromorphic modular forms and the space $\qm^{*}_{-}$ is generated by nonpositive weight derivatives of nonpositive weight meromorphic quasi-modular forms. It can be deduced from Laphier--El Gradechi formula~\eqref{eq:dfgb} that any product on each space can be rewritten as linear combination of iterated derivatives of Rankin--Cohen brackets. In particular, we get

\begin{proof}{Proof of Theorem~\ref{thm:fg}}
The result follows immediately from Laphier--El Gradechi formula~\eqref{eq:dfgb}. The condition $k+l\geq 2$ or $k+l+2n\leq 0$ is automatically satisfied in $\qm_{+}^{*}$ and $\qm_{-}^{*}$ respectively.
\end{proof}

\begin{rmk}
In fact, the Rankin--Cohen brackets can be extended to meoromorphic quasi-modular forms using $D$ and $\vartheta$. Let $f\in\qm_{k}^{*,\,\leq s}$ and $g\in\qm_{l}^{*,\,\leq t}$ be two meoromorphic quasi-modular forms. Their (Serre--)Rankin--Cohen brackets can be defined as
$$[f,g]_{n}=\sum_{j=0}^n(-1)^j\binom{n+k+s-1}{j}\binom{n+l+t-1}{n-j}D^{n-j}fD^jg,$$
$$\textrm{Se}\,[f,g]_{n}=\sum_{j=0}^n(-1)^j\binom{n+k-1}{j}\binom{n+l-1}{n-j}\vartheta^{n-j}f\vartheta^jg,$$
where $[f,g]_{n}$ and $\textrm{Se}\,[f,g]_{n}$ are meromorphic quasi-modular forms in $\qm_{k+l+2n}^{*,\,\leq s+t}$.
\end{rmk}

\section{Fourier coefficients of meromorphic quasi-modular forms}\label{sec:6}

Let us denote by 
\[\mathcal{F}=\{\tau\in \mch\cup\{\infty\}\,|\,-1/2\leq \RE(\tau)<1/2\}\]
the standard fundamental domain of the translation $\tau\mapsto\tau+1$ in $\Hcal\cup\{\infty\}$. We set also the trimmed fundamental domain for $t_0>0$ by $$\mathcal{F}_{t_0}=\{\tau\in \mcf\,|\,\IM \tau\geq t_0\}.$$
Let $f$ be a meromorphic function on $\mathcal{F}$. For any pole $\alpha$ of $f$ with order $\ord_f(\alpha)$, let
 \begin{equation}
  \PP_{\infty}(f)(\tau)=\sum_{n<0}a_f(n)q^n
 \end{equation}
be the principle part of the Fourier expansion of $f$ at $\alpha=\infty$ and
\begin{equation}
    \PP_{\alpha}(f)(\tau)=\sum_{m=1}^{\ord_f(\alpha)}c_{f,\alpha}(m)\frac{1}{(\tau-\alpha)^m}
\end{equation}
be the principle part of the Laurent expansion of $f$ at $\tau=\alpha$. We put also
$$P_{\infty}(f)=\sum_{n<0}a_f(n)q^n,$$
and
$$P_{\alpha}(f)=\sum_{m=1}^{\ord_f(\alpha)}\f{(-2\pi i)^m}{(m-1)!}c_{f,\alpha}(m)\Li_{1-m}(\be(\tau-\alpha)),$$
where $\be(\tau)=e^{2\pi i\tau}$. Note that the expansion $\eqref{eq:limz}$ in Appendix shows that $P_{\alpha}(f)$ has the same principle part at $\tau=\alpha$ as $f$ does.

We have the following estimation on Fourier coefficients of meromorphic quasi-modular form
\begin{prop}
Let $f=\sum_{n\gg-\infty} a_{f}(n)q^n$ be a meromorphic quasi-modular form of weight $k$ with poles and principle parts as described above. 
Then we have
$$a_f(n)=\sum_{\substack{\alpha\in\mathcal{F}_{t_0},\\1\leq m\leq \ord_{f}(\alpha)}}\f{(-2\pi i)^m}{(m-1)!}c_{f,\alpha}(m)n^{m-1}e^{-2\pi i n\alpha}+O(e^{2\pi n t_0}).$$
\end{prop}
\begin{proof}
 Let $\alpha_{1},\dots,\alpha_{l}$ be all the poles of $f$ in $\mathcal{F}_{t_0}$. Removing all the principle parts with all these kinds of $P_{\alpha}(f)$, we have finally a holomorphic function in $\mathcal{F}_{t_0}$
$$\tilde{f}(\tau)=f(\tau)-P_{\alpha_1}(f)-\dots-P_{\alpha_l}(f).$$
So $\tilde{f}$ is bounded in $\mcf_{t_0}$, say by $C$. Note the function $\tilde{f}(\tau)$ is holomorphic within the closure of the domain $\mathcal{F}_{t_0}$. Then using Cauchy integral formula at infinity, the $n$-th coefficient of $\tilde{f}(\tau)$ is bounded by 
\begin{equation}\label{eq:esti}
    \left|\int_{it_0}^{it_0+1}\tilde{f}(\tau)e^{-2\pi in\tau}d\tau\right|\leq e^{2\pi int_0}\int_0^1\left|\tilde{f}(t+it_0)\right|dt\leq Ce^{2\pi nt_0}.
\end{equation}
On the other hand, the $n$-th coefficient of the polylogarithm function $P_{\alpha}(f)$ is 
\[\sum_{m=1}^{\ord_f(\alpha)}\f{(-2\pi i)^m}{(m-1)!}c_{f,\alpha}(m)n^{m-1}e^{-2\pi in\alpha}.\]
Combining with the estimation $\eqref{eq:esti}$, we get the desired result.
\end{proof}

Weakly holomorphic quasi-modular forms have much less growth than meromorphic quasi-modular forms. One has
\begin{prop}
Let $f=\sum_{n\geq -n_0}a_f(n)q^n\in \qm_k^{!,\,p}$ be a weakly holomorphic quasi-modular form with $n_0>0$. Then for any $\varepsilon>0$, we have 
\[a_f(n)\ll e^{(4\pi+\varepsilon)\sqrt{n_0n}}.\]
\end{prop}
\begin{proof}
When $F$ is a weakly holomorphic modular form, then we have the following estimation~\cite{Ra38} on the Fourier coefficients of $F$
\begin{equation}\label{eq:afn}
    a_F(n)\ll e^{(4\pi+\epsilon)\sqrt{n_0n}}.
\end{equation}
Theorem~\ref{thm:fE2} shows that every weakly holomorphic quasi-modular form is of the form
\[f=F_0+F_1E_2+\cdots+F_pE_2^p.\]
Note that the order of $F_i$ at infinity is at most $n_0$ for any $i=0,1,\cdots,p$. Combining the estimation $\eqref{eq:afn}$, we get
\[a_f(n)\ll e^{(4\pi+\varepsilon)\sqrt{n_0n}},\]
since $\sigma(n)\ll n^{1+\epsilon}$ for any $\epsilon>0$. 
\end{proof}
\begin{rmk}
In fact, we can get a more accurate estimation by using the Circle Method to weakly holomorphic quasi-modular forms
\[a_f(n)\ll n^{\frac{2k-3}{4}}e^{4\pi\sqrt{n_0n}}.\]
\end{rmk}

\section{Regularized integrals of meromorphic functions}\label{sec:7}
At the start of this section, we recall the $L$-functions of modular forms. For a cusp form $f$ of weight $k$ on $\Hcal$, the completed $L$-function of $f$ is just the Mellin transform of $f$
$$\Lambda(f,s)=\int_0^{\infty}f(it)t^{s-1}dt.$$
This completed $L$-function is connected with the Dirichlet series of $f$ in the following way
$$\Lambda(f,s)=\f{\Gamma(s)}{(2\pi)^s}\sum_{n=1}^{\infty}\f{a_f(n)}{n^s}.$$

However in general, a meromorphic quasi-modular form may have exponential growth at cusps and polynomial growth at poles. To overcome this, we now construct the regularized integrals of meromorphic functions. This regularization procedure can be divided into two parts, the regularization at infinity (and hence at $0$) and the regularization at positive real numbers.

Following \cite{BDE17}, under the assumption that $f$ has at most linear exponential growth at infinity, we give the definition of regularized integral of $f$.
\begin{defn}\label{defn:infinity}
Let $f(t)$ be an analytic function with at most linear exponential growth for large $t\in\R_{>0}$. If the integral \[\int_{t_0}^{\infty}e^{-wt}f(t)dt\]
has a continuation to $w=0$, then the \emph{regularized integral} of $f$ is defined to be 
\[\int_{t_0}^{\infty,\ast}f(t)dt\coloneqq\left[\int_{t_0}^{\infty}e^{-wt}f(t)dt \right]_{w=0}.\]
We use the notation $*$ to indicate a regularized integral.
\end{defn}
Similarly, if $f(1/t)$ has at most linear exponential growth at the cusp $0$, then we can define the regularized integral of $f$ at $0$ using the reflection $t\mapsto 1/t$. This is to say,
\[\int_{0}^{t_0,\ast}f(t)dt\coloneqq\int_{t_0^{-1}}^{\infty,\ast}\frac{1}{t^2}f\left(\frac{1}{t}\right)dt.\]

For integrals of meromorphic functions near real positive poles, we will use Hadamard regularization. The idea of regularizing a divergent integral can be traced back to Cauchy. Precise definition of such regularized integrals was firstly introduced by Hadamard~\cite{Had32} in his study of Cauchy problem for differential equations of hyperbolic type. The interpretation of Hadamard regularization using meromorphic continuation was due to Riesz~\cite{Riesz38, Riesz49}. Various theories and generalization of Hadamard regularization can be found in the later literature.  Gelfand--Shilov~\cite{GS62} formalized Hadamard regularization in the framework of generalized functions.
Afterwards, the concept of generalized functions was extended to hyperfunctions by Sato~\cite{Sato59}. Due to space constraints we will neglect the technical and theoretical details in this paper. The reader can find more precise presentations in the previous mentioned articles and the books~\cite{K88, K98}.

We can use meromorphic continuation to give the following definition.
\begin{defn}\label{def:regint}
Let $f(t)$ be a meromorphic function in a neighbourhood of $[a,b]$ with only one real positive pole $a<c<b$, then the \emph{regularized integral} of $f$ from $a$ to $b$ is defined as
$$\int_{a}^{b,*}f(t)dt\coloneqq\left[\int_a^b \abs{t-c}^s f(t) dt\right]_{s=0},$$
 where the suffix indicates the constant term in the Laurent expansion of $s$ at $0$.
\end{defn}

As already mentioned, there are different approaches of Hadamard regularization. The following proposition explains why they are actually equivalent.
\begin{prop}\label{prop:positivereals}
Let $f(t)$ be a meromorphic function in a neighbourhood of $[a,b]$ with only one real positive pole $a<c<b$ of order $n$. Let $F(t)=f(t)(t-c)^n$. The the following different approaches of regularization coincides
\begin{enumerate}[label=(\roman*)]
    \item The meromorphic continuation of the integral by Riesz
    $$\left[\int_a^b \abs{t-c}^s f(t) dt\right]_{s=0}.$$
    \item The integral of $f(t)$ in the sense of Sato's hyperfunction. Equivalently, the Cauchy principal valued integral by
    $$\frac{1}{2}\biggl(\int_{C^+}+\int_{C^-}\biggr) f(t)dt,$$
    where $C^{+}$ (resp. $C^{-}$) is a path from $a$ to $b$ above (resp. below) the real axis.
    \item The Hadamard finite part integral of $f(t)$
    $$\FP_{\varepsilon=0}\, \biggl(\int_a^{c-\varepsilon}+\int_{c+\varepsilon}^b\biggr) f(t) dt,$$
    where $\,\FP$ stands for the constant term in the Laurent expansion with respect to $\varepsilon$.
    \item The following integral in the sense of pairing the Schwartz distribution $\FP\,(x-c)^{-n}$ with $F(t)$
    \begin{align*}
\Bigl(\FP\,\frac{1}{(x-c)^n},F(t)\Bigr)&\coloneqq\sum_{i=0}^{n-1}\frac{(n-i-1)!}{n!}\biggl(\frac{F^{(i)}(a)}{(a-c)^{n-i}}-\frac{F^{(i)}(b)}{(b-c)^{n-i}}\biggr)\\
    &\quad\quad+\f{1}{n!}\PV \int_{a}^{b}\frac{F^{(n)}(t)}{t-c}dt, 
    \end{align*}
where $\PV$ stands for the Cauchy principal value of the integral.
    \item The following Cauchy principle value given by Sokhotski--Plemelj formula, i.e.
    $$\frac{1}{2}\sum_{\pm}\biggl(\lim_{u\to c\pm i0}\int_{a}^{b}\frac{F(t)}{(t-u)^n}dt\biggr),$$
    where $u$ tends to $c$ on both sides of real axis.
\end{enumerate}
\end{prop}
\begin{proof}
$(\rom{1})\Leftrightarrow(\rom{2})\Leftrightarrow(\rom{3})$. If $f(t)$ is holomorphic then the implication is immediate. So it suffices to check the function $f(t)=(t-c)^{-n}$. We may set the paths $C^{\pm}$ to be the paths agreed with the real axis but modified with small upper (resp. lower) semi-circles $S_{\varepsilon}^{\pm}$ at $c$ of radius $\varepsilon$.

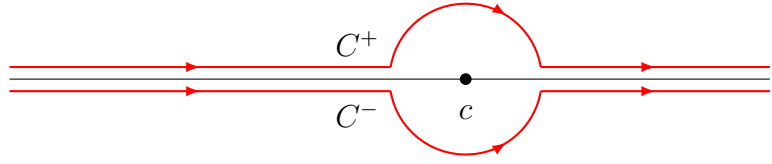
\begin{center}
\begin{tikzpicture}[scale=2.0, trim left=1cm]\label{fig:path}
\def\gap{0.08}
\def\pathscale{5}
\def\littleradius{0.5}
\def\littlangle{asin(\gap/\littleradius)}
\draw 
(0, 0) -- (5, 0*\pathscale);

\centerarc[red,thick,decoration={
	markings,
	mark=at position 0.7 with {\arrow{latex}}},postaction={decorate}](3,0)(180-\littlangle:0+\littlangle:\littleradius);
\draw[red,thick,decoration={
	markings,
	mark=at position 0.5 with {\arrow{latex}}},postaction={decorate}] (0, \gap) -- ({3-\littleradius*cos(\littlangle)},\gap) node[black, above left] {$C^{+}$};
\draw[red,thick,decoration={
	markings,
	mark=at position 0.5 with {\arrow{latex}}},postaction={decorate}] ({3+\littleradius*cos(\littlangle)},\gap) -- (\pathscale,\gap);
\centerarc[red,thick,decoration={
	markings,
	mark=at position 0.7 with {\arrow{latex}}},postaction={decorate}](3,0)(-180+\littlangle:0-\littlangle:\littleradius); 
\draw[red,thick,decoration={
	markings,
	mark=at position 0.5 with {\arrow{latex}}},postaction={decorate}] (0,-\gap) -- ({3-\littleradius*cos(\littlangle)},-\gap)  node[black, below left] {$C^{-}$};
\draw[red,thick,decoration={
	markings,
	mark=at position 0.5 with {\arrow{latex}}},postaction={decorate}] ({3+\littleradius*cos(\littlangle)},-\gap) -- (\pathscale,-\gap);
\filldraw [black] (3,0) circle (1pt) node[anchor=north] at (3,-0.1) {$c$};
\end{tikzpicture}
\end{center}
For $\RE(s)\ll 0$, the meromorphic continuation gives us
\begin{align*}
\int_{c-\varepsilon}^{c+\varepsilon}  \abs{t-c}^s f(t) dt &=(-1)^n\int_{c-\varepsilon}^{c}(c-t)^{s-n}dt+\int_{c}^{c+\varepsilon}(t-c)^{s-n}dt\\
&=(1+(-1)^n)\frac{\varepsilon^{s+1-n}}{s+1-n}.
\end{align*}
Hence, the constant term in the Laurent expansion at $s=0$ is
\begin{equation}\label{eq:epsinodd}
\left[\int_{c-\varepsilon}^{c+\varepsilon} \abs{t-c}^s f(t) dt\right]_{s=0}=
\begin{cases}	
2\,\varepsilon^{1-n}/(1-n) & n \text{ even}\\
0 & n \text{ odd}\end{cases}.
\end{equation}
Meanwhile, the integrations along the two small semi-circles $S_{\varepsilon}^{\pm}$ give
\begin{align*}
 \f{1}{2}\left(\int_{S^{+}_{\varepsilon}}+\int_{S^{-}_{\varepsilon}}\right)\f{1}{(t-c)^{n}}dt
&=\f{i}{2}\left(\int_{\pi}^{0}\varepsilon^{1-n}\f{d\theta}{e^{i(n-1)\theta}}+\int_{-\pi}^{0}\varepsilon^{1-n}\f{d\theta}{e^{i(n-1)\theta}}\right)
\end{align*}
An elementary calculation shows that this integral coincides with $\eqref{eq:epsinodd}$. This implies that the meromorphic continuation yields the same result as hyperfunction. 

Observe that the above integrals have always finite part $0$ with respect to $\varepsilon$. It follows that the Hadamard finite part integral has the same value as the integral of hyperfunction. These prove that $(\rom{1})$, $(\rom{2})$ and $(\rom{3})$ are equal.

$(\rom{3})\Leftrightarrow(\rom{4})$. Integrating by parts, for any testing function $\phi$ we get
\begin{align*}
    \int_{a}^{b}\frac{\phi(t)}{(t-c)^n}dt=&\sum_{i=0}^{n-1}\frac{(n-i-1)!}{n!}\biggl(\frac{\phi^{(i)}(a)}{(a-c)^{n-i}}-\frac{\phi^{(i)}(b)}{(b-c)^{n-i}}\biggr)\\
    &+\sum_{i=0}^{n-1}\frac{\phi^{(i)}(c)}{i!}\frac{1-(-1)^{n-i}}{(n-i)\varepsilon^{n-i}}+\f{1}{n!}\biggl(\int_{a}^{c-\varepsilon}+\int_{c+\varepsilon}^{b}\biggr)\frac{\phi^{(n)}(t)}{t-c}dt.
\end{align*}
The last term is a convergent integral with Cauchy principal value as $\varepsilon\to 0$. Thus
\begin{align*}
 \Bigl(\FP\,\frac{1}{(x-c)^n},\phi(t)\Bigr)=&\sum_{i=0}^{n-1}\frac{(n-i-1)!}{n!}\biggl(\frac{\phi^{(i)}(a)}{(a-c)^{n-i}}-\frac{\phi^{(i)}(b)}{(b-c)^{n-i}}\biggr)\\
 &+\f{1}{n!}\lim_{\varepsilon\to 0}\biggl(\int_{a}^{c-\varepsilon}+\int_{c+\varepsilon}^{b}\biggr)\frac{\phi^{(n)}(t)}{t-c}dt
\end{align*}
is the finite part with respect to $\varepsilon$. This shows $(\rom{3})\Leftrightarrow(\rom{4})$.

$(\rom{4})\Leftrightarrow(\rom{5})$. This parts follows closely with Fox~\cite{Fox} and Gelfand--Shilov~\cite{GS62}. Again using integration by parts, for any $u$ not in $[a,b]$, one has
\begin{align*}
      \int_{a}^{b}\frac{F(t)}{(t-u)^n}dt=&\sum_{i=0}^{n-1}\frac{(n-i-1)!}{n!}\biggl(\frac{F^{(i)}(a)}{(a-u)^{n-i}}-\frac{F^{(i)}(b)}{(b-u)^{n-i}}\biggr)\\
      &+\f{1}{n!}\int_{a}^{b}\frac{F^{(n)}(t)}{t-u}dt.  
\end{align*}
Let $u\to c\pm i0$ from both sides of real axis, by the Sokhotski--Plemelj formula of Cauchy principal valued integral we obtain $(\rom{4})\Leftrightarrow(\rom{5})$.
\end{proof}

In general, let $f(t)$ be a function which has a finite number of positive real poles and has at most linear exponential growth at $0$ and infinity. Consider finitely many open intervals
$$(a_1,a_2),(a_2,a_3)\dots,(a_{n-1},a_n).$$
Suppose that all poles are contained in these intervals and each interval contains exactly one isolated pole. On every interval, we use the previous approaches from Proposition~\ref{prop:positivereals} to get a regularized integral. Moreover, on the intervals
$$(0,a_1),(a_n,\infty),$$
we assign to them the regularized integrals from Definition~\ref{defn:infinity}. At last, we sum up them all. This gives the \emph{regularized integral of $f$ on $(0,\infty)$}, written again as
$$\int_{0}^{\infty,*}f(t)dt.$$
It is clear that the above definition is independent of the choice of intervals.

\section{$L$-function of meromorphic quasi-modular forms}\label{sec:8}
We first define the $L$-function of a meromorphic quasi-modular form through the regularized integral defined in Section \ref{sec:7}. 
\begin{defn}
Let $f$ be a meromorphic quasi-modular form. Then we define its \emph{complete $L$-function} by
\[\Lambda(f,s)=\int_0^{\infty,\ast}f(it)t^{s-1}dt.\]
The \emph{Dirichlet $L$-function} associated to $f$ is defined as
\[L(f,s)=\frac{(2\pi)^s}{\Gamma(s)}\Lambda(f,s).\]
\end{defn}

In the following, we will give an explicit formula for $\Lambda(f,s)$ for any meromorphic quasi-modular form $f$. We first consider the regularized integral of meromorphic function $f$ at infinity. 
\begin{lem}\label{lem:hol}
Let $f$ be a meromorphic function in a neighbourhood of the half-strip $\mathcal{F}_{t_0}$ with only pole at infinity. Suppose its Fourier expansion at infinity is given as
$f(\tau)=\sum_{n\geq -n_0}a_f(n)q^n$. 
Then the regularized integral of $f(it)t^{s-1}$ exists and defines a meromorphic function in $s$. More precisely, we have
\[\int_{t_0}^{\infty,\ast}f(it)t^{s-1}dt=-\frac{a_f(0)t_0^s}{s}+\sum_{n\neq 0}\frac{a_f(n)\Gamma(s,2\pi nt_0)}{(2\pi n)^s}.\]
\end{lem}

\begin{proof}
To avoid problems on negative real axis, following~\cite{BDE17}, we take only one branch of the incomplete gamma function with the branch cut to be the ray $\{re^{i\theta}\,|\,r\in\R_{>0}\}$ , where $\theta\in (\pi,\f{3}{2}\pi)$ is a fixed angle. It is easy to see that when $\RE(w)>2\pi n_0$, the integral \[\int_{t_0}^{\infty}\sum_{n\neq 0}a_f(n)e^{-(w+2\pi n)t}t^{s-1}dt\]
is absolutely convergent for any $s\in \mbc$. 

Since $f$ is holomorphic in a neighbourhood of the half-strip $\mathcal{F}_{t_0}$, its Fourier coefficients satisfy $a_f(n)=O(e^{2\pi nt_0})$. This ensures that the value of the above integral is the absolutely convergent sum
\[\sum_{n\neq 0}\frac{a_f(n)\Gamma(s,(2\pi n+w)t_0)}{(2\pi n+w)^s}=\sum_{n>0}+\sum_{n<0}\,,\]
in view of $\Gamma(s,2\pi nt_0)\sim(2\pi nt_0)^{s-1}e^{-2\pi n t_0}$.

We can see that the partial sum $\sum_{n>0}$ defines a holomorphic function of $(w,s)$ with $\RE(w)>-2\pi$ and $s\in\C$. The sum $\sum_{n<0}$ is a finite sum, it can be continued to a holomorphic function of $(\omega,s)$ in the open domain 
$$\C\,\backslash\bigcup_{m=1}^{n_0}\{m+re^{i\theta}\,|\,r\in\R_{\geq0}\}\times\C.$$
Hence both parts can be extend to a holomorphic function of $s$ in a neighbourhood of $w=0$.

We only need to deal with the term $n=0$. When $w=0$ and $\RE(s)<0$, the integral has well-defined value $-t_0^s/s$ (cf.~\cite[Remark after Prop. 3.3]{BDE17}). This extends to a meromorphic function to the whole complex plane in $s$.

At last we remark that the above evaluation is independent of the choice of $\theta$.
\end{proof}

To give the precise formula of regularized integral at positive reals, we will follow the method from  McGady~\cite{Mc19}, whose idea is to remove all the poles with polylogarithm functions. The succeeding calculation deals with the regularized integrals of polylogarithm functions first. 

Let $m\geq 0$ be an integer. For $s,\alpha\in\mbc$ with $-1/2\leq \RE(\alpha)<1/2$, we define the regularized integral
\[J_{m}(s,\alpha)= \int_{0}^{\infty,\ast}\Li_{-m}(\be(it-\alpha))t^s\f{dt}{t}.\]

\begin{lem}\label{lem:jnsa}
Let $\RE (s)>0$ and $m$ be a positive integer. If $\IM(\alpha)<0$, then we have
\[J_m(s,\alpha)=\f{\Gamma(s)}{(2\pi)^s}\Li_{s-m}(e^{-2\pi i\alpha}).\]
If $\IM(\alpha)>0$, we have
\begin{equation}\label{eq:j1mam}
    \begin{split}
        J_{m}(s,&\,\alpha)=\frac{e^{i\pi(s-m)}}{(2\pi )^m}\frac{\Gamma(s)}{\Gamma(s-m)}\zeta(1-s+m,\lfloor \RE\alpha\rfloor +1-\alpha)\\
        &-\frac{\Gamma(s)e^{i\pi(s-m)}}{(2\pi)^s}\Li_{s-m}(e^{2\pi i\alpha})+\delta_{\RE\alpha=0}\frac{\Gamma(s)}{\Gamma(s-m)}\frac{i^s(-\alpha)^{s-m-1}}{2(2\pi i)^m}.
    \end{split}
\end{equation}
\end{lem}
\begin{proof}
Suppose that $\alpha$ is given with $\IM (\alpha)<0$, then $|\be(it-\alpha)|<1$, so we have the convergent integral
\begin{equation}\label{eq:jnsa}
\begin{split}
    &J_{m}(s,\alpha)=\int_{0}^{\infty}\sum_{n=1}^{\infty}\f{e^{-2\pi n(y+i\alpha)}}{n^{-m}}t^s\f{dt}{t}\\
    =&\f{\Gamma(s)}{(2\pi)^s}\sum_{n=1}^{\infty}\f{e^{-2\pi i n\alpha}}{n^{s-m}}=\f{\Gamma(s)}{(2\pi)^s}\Li_{s-m}(e^{-2\pi i\alpha}).
    \end{split}
\end{equation}
Note both sides extend to a holomorphic function of $\alpha$ except only when $\alpha$ on the imaginary axis. Thus it holds for all $\RE(\alpha)\neq 0$.

When $\IM(\alpha)>0$,  if $\alpha$ not on the imaginary axis, the formula $\eqref{eq:j1mam}$ just follows from rewriting~\eqref{eq:jnsa} with the reflection formula of polylogarithm~\eqref{eq:reflcli} in the Appendix.

The difficulty arises as $\alpha=ai$ is on the imaginary axis where $a\in\mbr_{>0}$, where we encounter a Hadamard regularized integral. In this case we may rewrite the integral as
$$J_{m}(s,\alpha)=\int_{0}^{\infty,*}\Li_{-m}(\be(\tau-\alpha))\left(\f{\tau}{i}\right)^s\f{d\tau}{\tau}.$$
We recall that when $m$ is a positive integer, the polylogarithm $\Li_{-m}(z)$ is rational function. So the integrand is in fact a rational function of $e^{2\pi i\alpha}$. By the Sokhotski--Plemelj formula $(\rom{5})$ in Proposition~\ref{prop:positivereals}, the value of $J_{m}(s,\alpha)$ on imaginary axis should be the mean value of limits as $\RE (\alpha)$ tends to $0$ on left side and right side of imaginary axis. Moreover, when $z\in [1,\infty)$, by the reflection formula~\eqref{eq:reflcli}, we get
\begin{equation}\label{eq:refep}
    \lim_{\epsilon\to 0^+}\Li_s(ze^{2\pi i\epsilon})-\Li_s(ze^{-2\pi i\epsilon})=\frac{2\pi i}{\Gamma(s)}(\ln z)^{s-1}.
\end{equation}
So by combining $\eqref{eq:jnsa}$ and $\eqref{eq:refep}$, we have
\begin{align*}
   J_m(s,\alpha)&=\f{1}{2}\f{\Gamma(s)}{(2\pi)^s}\lim\limits_{\varepsilon\to 0^+}\left(\Li_{s-m}(e^{-2\pi i(\alpha+\varepsilon)})+\Li_{s-m}(e^{-2\pi i(\alpha-\varepsilon)})\right)\\
   &=\f{1}{2}\f{\Gamma(s)}{(2\pi)^s}\left(2\Li_{s-m}(e^{2\pi a})+\f{2\pi i}{\Gamma(s-m)}(2\pi a)^{s-m-1}\right).
\end{align*}
Using the reflection formula again, we get
\begin{align*}
J_m(s,\alpha)&=-\frac{e^{i\pi(s-m)}\Gamma(s)}{(2\pi )^s}\Li_{s-m}(e^{-2\pi a})+\frac{i^s}{(2\pi i)^{m}}\frac{\Gamma(s)}{\Gamma(s-m)}\zeta(1-s+m,1-ia)
\\&+\frac{i^s}{2(2\pi i)^{m}}\frac{\Gamma(s)}{\Gamma(s-m)}(-ia)^{s-m-1}.
\end{align*}
\end{proof}

More generally, we will encounter the following integral. Let $m\geq 0$ be an integer, $t_0>0$ be a real number and $s,\alpha\in \mbc$, we define
\[G_{m}(s,\alpha,t_0)=\int_{t_0}^{\infty,*}\Li_{-m}(\be(it-\alpha))t^s\f{dt}{t}.\]

\begin{lem}\label{lem:mmsat}
Let $m\geq 0$ be an integer and $t_0$ be a positive real number. Then the function $G_m(s,\alpha,t_0)$ extends to an entire function for all $s\in\C$ and 
$$G_m(s,\alpha,t_0)=G_{1,m}(s,\alpha,t_0)+G_{2,m}(s,\alpha,t_0).$$
When $\IM(\alpha)<t_0$, we have 
\[G_{1,m}(s,\alpha,t_0)=\f{1}{(2\pi)^{m}}\sum_{n=0}^{\infty}\f{e^{-2\pi in \alpha}\Gamma(s,2\pi nt_0)}{(2\pi n)^{s-m}} \text{ and } G_{2,m}(s,\alpha,t_0)=0.\]
When $\IM(\alpha)>t_0$, we have
$$
G_{1,m}(s,\alpha,t_0)=-\frac{e^{-i\pi(s-m)}}{(2\pi)^m}\left(\delta_{m=0}\frac{t_0^s}{s}+\sum_{n=1}^{\infty}\frac{e^{2\pi in\alpha}\Gamma(s,-2\pi nt_0)}{(2\pi n)^{s-m}}\right)
$$
and
\begin{align*}
G_{2,m}&(s,\alpha,t_0)=\frac{i^s}{(2\pi i)^m}\frac{\Gamma(s)}{\Gamma(s-m)}\zeta(1-s+m,\lfloor\RE\alpha\rfloor+1-\alpha)\\
&+\f{(-1)^{m-1}}{(2\pi)^{s}}\frac{ 2\pi i}{\Gamma(1-s)}\Li_{s-m}(e^{2\pi i\alpha})+\delta_{\RE\alpha=0}\frac{i^s}{2(2\pi i)^m}\frac{\Gamma(s)}{\Gamma(s-m)}(-\alpha)^{s-m-1}.    
\end{align*}

\end{lem}
\begin{proof}

When $\IM \alpha<t_0$, we can integrate term-wisely. It is immediate that for any $s\in\C$
\[G_{m}(s,\alpha,t_0)=\sum_{n=1}^{\infty}\int_{t_0}^{\infty}\f{e^{-2\pi n(t+i\alpha)}}{n^{-m}}t^s\f{dt}{t}=\f{1}{(2\pi)^{m}}\sum_{n=1}^{\infty}\f{e^{-2\pi in\alpha}\Gamma(s,2\pi nt_0)}{(2\pi n)^{s-m}}.\]
Here the absolute convergence is guaranteed by $\Gamma(s,2\pi nt_0)\sim(2\pi nt_0)^{s-1}e^{-2\pi n t_0}$.

When $\IM(\alpha)>t_0$, we may assume that $\RE(s)>0$ first.  The integral is defined by Hadamard regularization and can not be computed directly. We first evaluate the following convergent integral
\[\int_{0}^{t_0}\Li_{-m}(\be(it-\alpha))t^s\f{dt}{t}.\]
When $m>0$, by the reflection formula $\eqref{eq:reflcli}$, this integral equals to 
\[(-1)^{m-1}\int_{0}^{t_0}\Li_{-m}(\be(\alpha-it))t^s\f{dt}{t}.\]
Hence we have 
\begin{equation}
    \begin{split}
        &\int_{0}^{t_0}\Li_{-m}(\be(\tau-\alpha))t^s\f{dt}{t}\\
=&(-1)^{m-1}\sum_{n=1}^{\infty}\int_{0}^{t_0}\f{t^se^{2\pi n(t+i\alpha)}}{n^{-m}}\f{dt}{t}\\
=&-e^{-i\pi(s-m)}\f{1}{(2\pi)^{m}}\sum_{n=1}^{\infty}\f{e^{2\pi in\alpha}\gamma(s,-2\pi n t_0)}{(2\pi n)^{s-m}}
    \end{split}
\end{equation}
where each term has exponential decay since $\gamma(s,-2\pi n t_0)\sim(-2\pi nt_0)^{s-1}e^{2\pi nt_0}$ as $n$ grows to infinity. When $m=0$, we have $\Li_0(x)=x/(1-x)$, so $\Li_0(x)=-\Li_0(1/x)-1$. In this case, the formula becomes 
\[\int_{0}^{t_0}\Li_{0}(\be(\tau-\alpha))t^s\f{dt}{t}=-e^{-i\pi s}\left(\frac{t_0^s}{s}+\sum_{n=1}^{\infty}\f{e^{-2\pi na}\gamma(s,-2\pi n t_0)}{(2\pi n)^{s}}\right)\]

To finish the proof, by combining with Lemma \ref{lem:jnsa}, we have to show that 
\begin{equation}\label{eq:eulerref}
    \begin{split}
        &e^{-i\pi(s-m)}\sum_{n\geq 1}\frac{e^{2\pi in\alpha}\gamma(s,-2\pi nt_0)}{n^{s-m}}-\Gamma(s)e^{i\pi(s-m)}\Li_{s-m}(e^{2\pi i\alpha})\\
    =&\frac{(-1)^{m-1}2\pi i}{\Gamma(1-s)}\Li_{s-m}(e^{2\pi i\alpha})-e^{-i\pi(s-m)}\sum_{n\geq 1}\frac{e^{2\pi in\alpha}\Gamma(s,-2\pi nt_0)}{n^{s-m}}.
    \end{split}
\end{equation}
From $\Li_{s-m}(e^{2\pi i\alpha})=\sum_{n\geq 1}n^{m-s}e^{2\pi in\alpha}$, we know the identity $\eqref{eq:eulerref}$ is equivalent to 
\[e^{-i\pi(s-m)}\Gamma(s)-e^{i\pi(s-m)}\Gamma(s)=\frac{(-1)^{m-1}2\pi i}{\Gamma(1-s)}.\]
But this is exactly the Euler's reflection formula.

For general $s\in\C$, we consider analytic continuation on both sides and thus obtain the same formula. Indeed, the function $G_{2,m}(s,\alpha,t_0)$ is meromorphic only when $m=0$. It has a unique single pole at $s=0$ with residue
$$\Res_{s=0}\,\zeta(1-s,\lfloor\RE\alpha\rfloor+1-\alpha)=-1.$$
However, this pole cancels with the term $\delta_{m=0}\,t_0^s/s$ in $G_{1,m}(s,\alpha,t_0)$, giving us an entire function $G_m(s,\alpha,t_0)$.
\end{proof}

Now we are able to give the explicit formula for the $L$-function. Choose any real positive number $t_0$. Suppose $f$ has poles
$\alpha_1,\cdots,\alpha_l$ in $\mathcal{F}_{t_0}-\{\infty\}$.
Put $$\tilde{f}(\tau)=f(\tau)-P_{\alpha_1}(f)-\dots-P_{\alpha_l}(f).$$
We define also
\[I(f,s,t_0)\coloneqq\sum_{j=1}^l\sum_{m=1}^{\ord_f \alpha_j}\frac{(-2\pi i)^m}{(m-1)!}c_{f,\alpha_j}(m)G_{m-1}(s,\alpha_j,t_0).\]

\begin{thm}\label{thm:lfunction}
Let $t_0$ be any real positive number. Let $f\in\qm_k^{\mero,\,p}$ be a meromorphic quasi-modular form with prescribed poles and principle parts as above. Let $f_1,\cdots,f_p$ be the component functions of $f$. Suppose that $\tilde{f}(\tau)=\sum_{n\gg -\infty} \tilde{a}_f(n)q^n$, then we have
\begin{align*}
    \Lambda(f,s)=&-\tilde{a}_f(0)\left(\frac{t_0^s}{s}+\frac{i^kt_0^{s-k}}{k-s}\right)+\sum_{n\neq 0}\tilde{a}_f(n)\left(\frac{\Gamma(s,2\pi n t_0)}{(2\pi n)^s}+\frac{i^k\Gamma(k-s,2\pi n/t_0)}{(2\pi n)^{k-s}}\right)\\
    &+I(f,s,t_0)+i^kI(f,k-s,t_0^{-1})+\sum_{r=1}^p\left(-\frac{i^{k-r}\tilde{a}_{f_r}(0)t_0^{k-r-s}}{k-r-s}\right.\\
    &+\left.\sum_{n\neq 0}\frac{i^{k-r}\tilde{a}_{f_r}(n)\Gamma(k-r-s,\frac{2\pi n}{t_0})}{(2\pi n)^{k-r-s}} +i^{k-r}I(f_r,k-r-s,t_0\inv)\right).
\end{align*}
\end{thm}
\begin{proof}
We divide $\Lambda(f,s)$ into two parts:
\[\Lambda(f,s)= \int_{0}^{t_0,\ast}f(it)t^{s-1}dt+ \int_{t_0}^{\infty,\ast}f(it)t^{s-1}dt.\]
We first deal with the second part.
\begin{align*}
     \int_{t_0}^{\infty,\ast}f(it)t^{s-1}dt= \int_{t_0}^{\infty,\ast}\tilde{f}(it)t^{s-1}dt+\sum_{j=1}^l \int_{t_0}^{\infty,\ast}P_{\alpha_j}(f)(it)t^{s-1}dt.
\end{align*}
Since $\tilde{f}$ is holomorphic, by Lemma \ref{lem:hol}, we get 
\[ \int_{t_0}^{\infty,\ast}\tilde{f}(it)t^{s-1}dt=-\frac{\tilde{a}_f(0)t_0^s}{s}+\sum_{n\neq 0}\frac{\tilde{a}_f(n)\Gamma(s,2\pi nt_0)}{(2\pi n)^s}.\]
The integral of $P_{\alpha_j}(f)$ is shown in Lemma \ref{lem:mmsat} which gives 
\[ \int_{t_0}^{\infty,\ast}P_{\alpha_j}(f)(it)t^{s-1}dt=\sum_{m=1}^{\ord_{f} \alpha_j}\frac{(-2\pi i)^m}{(m-1)!}c_{f,\alpha_j}(m)G_{m-1}(s,\alpha_j,t_0).\]
For the first part, by changing the variable $t\to 1/t$, we get
\[ \int_{0}^{t_0,\ast}f(it)t^{s-1}dt= \int_{t_0\inv}^{\infty,\ast}f(i/t)t^{-s-1}dt.\]
Since $f$ is a meromorphic quasi-modular form,  we have the transformation
\begin{equation}\label{eq:fit}
    f(i/t)=\sum_{r=0}^pf_r(it)(it)^{k-r},
\end{equation}
by applying equation $\eqref{eq:slashf}$ with the inversion $\gamma=\begin{pmatrix}0&-1\\1&0\end{pmatrix}$. So we have 
\begin{align*}
    & \int_{0}^{t_0,\ast}f(it)t^{s-1}dt=\sum_{r=0}^p i^{k-r} \int_{t_0\inv}^{\infty,\ast}f_r(it)t^{k-r-s-1}dt\\
    &\quad=\sum_{r=0}^p\left(-\frac{i^{k-r}\tilde{a}_{f_r}(0)t_0^{k-r-s}}{k-r-s}+\sum_{n\neq 0}\frac{i^{k-r}\tilde{a}_{f_r}(n)\Gamma(k-r-s,\frac{2\pi n}{t_0})}{(2\pi n)^{k-r-s}}\right.\\
    &\qquad +i^{k-r}I(f_r,k-r-s,t_0\inv)\Bigg).
\end{align*}
Finally, we complete the proof by noting that $f_0=f$. 
\end{proof}

We are now ready to give the proof of Theorem~\ref{thm:lambdaf}.
\begin{proof}[Proof of Theorem~\ref{thm:lambdaf}]
The meromorphic continuation and residues follow directly from Theorem \ref{thm:lfunction}, since $I(f_r,s,t_0)$ is entire in $s$. So we only need to prove the functional equations. 

Lemma \ref{lem:qjf} shows that $f_m$ is also a meromorphic quasi-modular form of weight $k-2m$ and depth $p-m$ with components $f_m,\binom{m+1}{1}f_{m+1},\cdots,\binom{p}{p-m}f_p$. It is thus enough for us to show the functional equation of $f$ in $(\rom{1})$. The completed $L$-function of $f$ is
\begin{equation}\label{eq:fms}
    \begin{split}
    &\Lambda(f,s)= \int_{t_0}^{\infty,\ast}f(it)t^{s-1}dt+\int_{t_0\inv}^{\infty,\ast}f(i/t)t^{-s-1}dt\\
   &\quad= \int_{t_0}^{\infty,\ast}f(it)t^{s-1}dt+\int_{t_0\inv}^{\infty,\ast}\sum_{r=0}^{p}i^{k-r}f_{r}(it)t^{k-r-s-1}dt.
    \end{split}
\end{equation}
Here we use the transformation formula $\eqref{eq:fit}$ again. 

On the other hand, under the transform $t\mapsto 1/t$ the first integral becomes
\begin{equation}\label{eq:0t0}
\int_{t_0}^{\infty ,\ast}\,\sum_{r=0}^{p}i^{k-r}f_{r}(i/t)t^{r+s-k-1}dt=\int_0^{t_0\inv ,\ast}\,\sum_{r=0}^{p}i^{k-r}f_{r}(it)t^{k-r-s-1}dt.
\end{equation}
Combining equation $\eqref{eq:fms}$, we get 
\begin{align*}
    \Lambda(f,s)=&\left(\int_0^{t_0\inv ,\ast}+\int_{t_0\inv}^{\infty,\ast}\right)\sum_{r=0}^{p}i^{k-r}f_{r}(it)t^{k-r-s-1}dt\\
    =&\sum_{r=0}^{p}i^{k-r}\Lambda(f_{r},k-r-s).
\end{align*}
This proves the functional equation of $f$. 
\end{proof}

\begin{prop}~\label{prop:polesofL}
Let $f\in \qm_k^{\mero,\,p}$ be a meromorphic quasi-modular form. Then its Dirichlet $L$-function $L(f,s)$ is a meromorphic function for all $s\in\C$. It has only possible simple poles at positive integers within $k-p\leq s\leq k$.
\end{prop}
\begin{proof}
We note that $\Gamma(s)$ has a simple pole at nonpositive integers, so this proposition follows directly from Theorem \ref{thm:lambdaf}.
\end{proof}

The operator $D$ acts on the Fourier expansion of holomorphic quasi-modular form $f$ by
$$D=q\frac{d}{dq}:\sum_{n\geq 0} a_f(n)q^n\mapsto \sum_{n\geq 0} na_f(n)q^n.$$
This implies that the Dirichlet $L$-series of $Df$ is exactly the the shift of the original Dirichlet $L$-series of $f$. Actually, for meromorphic quasi-modular form, we can obtain the same result.
\begin{thm}\label{thm:shiftl}
Let $f$ be a meromorphic quasi-modular form. Then we have 
\[\Lambda(D^l f,s)=\frac{(s-l)_l}{(2\pi)^l}\Lambda(f,s-l),\]
and 
\[L(D^lf,s)=L(f,s-l).\]
\end{thm}
\begin{proof}
The above identities are nothing but integration by parts. Evidently it is enough for us to prove the case $l=1$. With integration by parts we have
\begin{align*}
    &\int_{t_0}^{\infty}(Df)(it)e^{-wt}t^{s-1}dt\\
    =&-\frac{1}{2\pi}f(it)e^{-wt}t^{s-1}\big|_{t_0}^{\infty}-\frac{1}{2\pi i}\int_{t_0}^{\infty}iwe^{-wt}f(it)t^{s-1}+\frac{s-1}{i}e^{-wt}f(it)t^{s-2}dt.
\end{align*}
When $w$ large enough, the first term equals to $\frac{1}{2\pi}f(it_0)t_0^{s-1}e^{-wt_0}$. Clearly, it has a holomorphic continuation to the whole plane in $w$ and its value at $w=0$ is just $\frac{1}{2\pi}f(it_0)t_0^{s-1}$. For the integral, by Lemma \ref{lem:hol}, it has a holomorphic continuation to a neighbourhood of $w=0$. So we get
\[\int_{t_0}^{\infty,\ast}(Df)(it)t^{s-1}dt=\frac{1}{2\pi}f(it_0)t_0^{s-1}+\frac{s-1}{2\pi}\int_{t_0}^{\infty,\ast}f(it)t^{s-2}dt.\]
Another way to see this is using the precise formula in Lemma~\ref{lem:hol} and the recurrence relation~\eqref{eq:recurrence}. We can deal with regularized integrals at positive real poles and $0$ in the same way with integration by parts. At last we get 
\[\int_0^{\infty,\ast}(Df)(it)t^{s-1}dt=\frac{s-1}{2\pi}\int_{0}^{\infty,\ast}f(it)t^{s-2}dt.\]
This gives the identity for $\Lambda(f,s)$. The identity for $L(f,s)$ then follows directly after $\Gamma(z+1)=z\Gamma(z)$. 
\end{proof}

If $f$ is a meromorphic modular form, the formula of its $L$-function is much simpler. 
\begin{cor}\label{cor:funeq}
Let $f\in\mcm_k^{\mero}$ be a meromorphic modular form of weight $k$. Then the $L$-function of $f$ is 
\begin{align*}
    \Lambda(f,s)=&-\tilde{a}_f(0)\left(\frac{t_0^s}{s}+\frac{i^kt_0^{s-k}}{k-s}\right)+I(f,s,t_0)+i^kI(f,k-s,t_0^{-1})\\
    +&\sum_{n\neq 0}\tilde{a}_f(n)\left(\frac{\Gamma(s,2\pi n t_0)}{(2\pi n)^s}+\frac{i^k\Gamma(k-s,2\pi n/t_0)}{(2\pi n)^{k-s}}\right)
\end{align*}
Moreover, it satisfies the following functional equation
\[\Lambda(f,s)=i^k\Lambda(f,k-s).\]
\end{cor}

\begin{rmk}
In particular, when $f\in S_k^{!}$ is a weakly holomorphic cusp form, we obtain
$$\Lambda(f,s)=\sum_{n\neq 0}a_f(n)\left(\frac{\Gamma(s,2\pi n t_0)}{(2\pi n)^s}+\frac{i^k\Gamma(k-s,2\pi n/t_0)}{(2\pi n)^{k-s}}\right).$$
This computation coincides with Theorem $2.2$ in \cite{BFK14}.
\end{rmk}

\section{Vanishing $L$-values of meromorphic quasi-modular forms}\label{sec:9}
In this section, we give some vanishing results of certain special $L$-values of meromorphic quasi-modular forms.
\begin{prop}\label{prop:lfs0}
Let $f\in \qm_k^{\mero,\,p}$ be a meromorphic quasi-modular form of weight $k$. Let $s$ be a negative integer, then
\begin{enumerate}[label=(\roman*)]
    \item If $k\leq 0$, then the Dirichlet $L$-function $L(f,s)$ is always entire in $s$. Moreover, when either $s<k-p$ or $k<s<0$,  we have 
    $$L(f,s)=0.$$
    \item If $k\geq 2$, when $s<k-p$, we have
    $$L(f,s)=0.$$
\end{enumerate}
\end{prop}

\begin{proof}
If $k\leq 0$, the function $\Lambda(f,s)$ has only possibly simple poles at $k-p\leq s\leq k\leq 0$ by Theorem~\ref{thm:lambdaf}. Meanwhile, $\Gamma(s)$ also has a simple pole at negative integers. So the Dirichlet $L$-function $L(f,s)$ is entire and $L(f,s)=0$ whenever $s<k-p$ or $k<s<0$.

If $k\geq 2$, the function $\Lambda(f,s)$ has no poles for $s<k-p$ but $\Gamma(s)$ has pole at negative integer $s$, so $L(f,s)$ always vanishes if $s<k-p$.
\end{proof}

\begin{cor}
Let $f\in \qm_k^{\mero,\,p}$ be a meromorphic quasi-modular form of weight $k\leq 0$. If $s$ is a positive integer with $2\leq s\leq \abs{k}$, then we have 
\[L(D^{1-k}f,s)=0.\]
\end{cor}
\begin{proof}
By Theorem \ref{thm:shiftl} and Proposition \ref{prop:lfs0}, when $2\leq s\leq -k$, we have
\[
L(D^{1-k}f,s)=L(f,s+k-1)=0.
\qedhere
\]
\end{proof}

\begin{rmk}
This corollary shows that some periods of the meromorphic quasi-modular form $D^{1-k}f$ vanish. When $f$ is a weakly holomorphic modular form, this recovers Theorem $2.5$ in \cite{BFK14}.
\end{rmk}

\begin{prop}
The period of a meromorphic modular form of weight $2$ is always vanishing, i.e. for any meromorphic modular form $f\in \mcm_2^{\mero}$, 
\[L(f,1)=\Lambda(f,1)=0.\]
\end{prop}
\begin{proof}
Applying the functional equation in Corollary~\ref{cor:funeq} at weight $2$, we get
\[\Lambda(f,1)=-\Lambda(f,1).\]
So the period $L(f,1)=\Lambda(f,1)=0$.
\end{proof}

\appendix
\section{Special functions}
\subsection{Incomplete Gamma functions} We give some basic properties of incomplete gamma functions. Let $s,z\in \mbc$ with $\RE (s)>0$ and $\abs{\arg(z)}<\pi$, the upper gamma function is defined by 
\[\Gamma(s,z)=\int_z^{\infty}t^{s-1}e^{-t}dt,\]
whereas the lower incomplete gamma function is defined by 
\[\gamma(s,z)=\int_0^{z}t^{s-1}e^{-t}dt.\]
The two functions are connected by 
$$\Gamma(s)=\Gamma(s,z)+\gamma(s,z).$$
Integration by parts yields the following recurrence relation for incomplete gamma function
\begin{equation}\label{eq:recurrence}
   \Gamma(s+1,z)=s\,\Gamma(s,z)+z^{s}\,e^{-z}.
\end{equation}
Both incomplete gamma functions can be extended with respect to both $s$ and $z$. The function $\Gamma(s,z)$ extends to a multi-valued function in $z$ with a branch point at $z=0$, and is holomorphic in each sector. If $z\neq 0$, then $\Gamma(s,z)$ is always an entire function in $s$.

For $(s,z)\in\C^2$, the incomplete gamma function has exponential growth asymptotic expansion $$\Gamma(s,z)\sim z^{s-1}e^{-z}\left(1+\f{s-1}{z}+\f{(s-1)(s-2)}{z^2}+\dots\right)$$
as $|z|\to \infty$ (see~\cite[Section 8.11 (i)]{NIST}).

\subsection{Hurwitz zeta function} Let $s,a\in\mbc$ with $\RE(s)>1$ and $a\neq 0,-1,-2,\cdots$. The Hurwitz function $\zeta(s,a)$ is defined by series expansion
\[\zeta(s,a)=\sum_{n=0}^{\infty}\frac{1}{(n+a)^s}.\] The function
$\zeta(s,a)$ has a meromorphic continuation to the whole $s$-plane. It has only a simple pole at $s=1$ with residue $1$.

\subsection{Polylogarithm} Let $s,z\in \mbc$ with $|z|<1$. The polylogarithm function is defined to be 
\[\Li_s(z)=\sum_{k=1}^{\infty}\frac{z^k}{k^s}.\]
It has an analytic continuation to a multi-valued holomorphic function in $z$ and an entire function in $s$. If $s$ is not a nonnegative integer, then the principle branch $\Li_{s}(z)$ is holomorphic in $z$ with branch cut along $[1,\infty)$.

When $n\in \mbn$, the function $\Li_{-n}(z)$ is a rational function in $z$ and has pole only at $z=1$. More precisely, we have 
\[\Li_{-n}(z)=\left(z\frac{\partial}{\partial z}\right)^n\frac{z}{1-z}.\]
Let $m\in\mbn$ be a positive integer, then the Laurent series of $\Li_{1-m}(\be(z))$ at $z=0$ is given by 
\begin{equation}\label{eq:limz}
    \Li_{1-m}(\be(z))=\frac{(m-1)!}{(-2\pi i z)^m}+\sum_{k=0}^\infty \frac{\zeta(1-m-k)}{k!}(2\pi i z)^k.
\end{equation}

The polylogarithm is related to the Hurwitz zeta function by
\[\Li_s(z)=\frac{\Gamma(1-s)}{(2\pi)^{1-s}}\left(i^{1-s}\zeta\left(1-s,\frac{1}{2}+\frac{\ln(-z)}{2\pi i}\right)+i^{s-1}\zeta\left(1-s,\frac{1}{2}-\frac{\ln(-z)}{2\pi i}\right)\right).\]
When $z\notin (0,1]$, we have the reflection formula of polylogarithm
\begin{equation}\label{eq:reflcli}
    \Li_s(z)+\Li_s(1/z)=\frac{(2\pi i)^s}{\Gamma(s)}\zeta\left(1-s,\,\frac{1}{2}+\frac{\ln(-z)}{2\pi i}\right).
\end{equation}

\section*{Acknowledgements}
The research of the second author was supported by Fundamental Research Funds for the Central Universities (Grant No. 531118010622).




\begin{thebibliography}{99}
\bibitem{BD19} A. Bhand, K. DeoShankhadhar, \emph{On Dirichlet series attached to quasimodular forms}, Journal of Number Theory (2019), Volume 202, pp. 91-106.
\bibitem{BDE17}K. Bringmann, N. Diamantis, S. Ehlen, \emph{Regularized inner products and errors of modularity}, International Mathematics Research Notices (2017), Issue 24, pp. 7420-7458.
\bibitem{BFK14} K. Bringmann, K. Fricke, Z. A. Kent, \emph{Special L-values and periods of weakly holomorphic modular forms}, Proceedings of the American Mathematical Society 142 (2014), pp. 3425-3439.
\bibitem{Coh75}H. Cohen, \emph{Sums involving the values at negative integers of L-functions of quadratic characters}, Mathematische Annalen (1975), Volume 217 No. 3, pp. 271–285.
\bibitem{CS179}H. Cohen, F. Str{\"o}mberg, \emph{Modular Forms: A Classical Approach}, Graduate Studies in Mathematics 179 (2017).
\bibitem{Gra13}A. M. El Gradechi, \emph{The Lie theory of certain modular form and arithmetic identities}, Ramanujan Journal 31, pp. 397–433.
\bibitem{Fox}C. Fox, \emph{A Generalization of the Cauchy Principal Value}, Canadian Journal of Mathematics (1957), Volume 9, pp. 110-117.
\bibitem{GS62}I. M. Gel'fand, G. E. Shilov, \emph{Generalized functions}, Volume 1, Moscow Fismatgiz (1962), [English translation, Academic Press (1964)].
\bibitem{Had32}J. Hadamard, \emph{Le probl{\`e}me de Cauchy et les {\'e}quations aux d{\'e}riv{\'e}es partielles lin{\'e}aires hyperboliques} (in French), Paris: Hermann \& Cie. (1932), p. 542
\bibitem{K88} A. Kaneko, \emph{Introduction to hyperfunctions},	Mathematics and its applications, Kluwer Academic Publishers (1988).
\bibitem{K98} Ram P. Kanwal, \emph{Generalized Functions: Theory and Technique}, Birkh{\"a}user Boston (1998).
\bibitem{KZ95}M. Kaneko, D. Zagier, \emph{A generalized Jacobi theta function and quasimodular forms, in: The Moduli Space of Curves}, Texel Island, 1994, in: Progr. Math., vol.129, Birkhäuser Boston, Boston, MA, 1995, pp.165–172.
\bibitem{Kuz75}N. V. Kuznetsov, \emph{A new class of identities for the Fourier coefficients of modular forms} (in Russian), Collection of articles in memory of Juri{\v \i} Vladimirovi{\v c} Linnik, Acta Arithmetica 27 (1975), pp. 505–519.
\bibitem{Lan08}D. Lanphier, \emph{Combinatorics of Maass–Shimura operators}, Journal of Number Theory (2008), Volume 128, Issue 8 , pp. 2467-2487.
\bibitem{LZ01}J. Lewis, D. Zagier, \emph{Period functions for Maass wave forms. I}, Annals of Mathematics (2001),  Volume 153, Issue 1, pp. 191-258.
\bibitem{LS22}S. L\"obrich, M. Schwagenscheidt, \emph{Locally harmonic maass forms and periods of meromorphic modular forms}, Transactions of the American Mathematical Society 375 (2022), 501-524.
\bibitem{Mc19}D. A. McGady, \emph{L-functions for Meromorphic Modular Forms and Sum Rules in Conformal Field Theory}, Journal of High Energy Physics (2019).
\bibitem{NIST} F. W. J. Olver et al., eds. \textit{NIST Handbook of Mathematical Functions},  Cambridge University Press (2010), pp. 173-191.
\bibitem{PZ20}V. Paşol, W. Zudilin, \emph{Magnetic (quasi-)modular forms}, \arxiv{2009.14609} (2020).
\bibitem{Ra38}H. Rademacher and H. S. Zuckerman, \emph{On the Fourier coeﬃcients of certain modular forms of positive dimension}, Annals of Mathematics (2) 39 (1938), No. 2, pp. 433–462.
\bibitem{Riesz38} M. Riesz, \emph{Int{\'e}grales de Riemann-Liouville et potentiels} (in French), Acta Scientiarum Mathematicarum 9 (1938), pp. 1–42.
\bibitem{Riesz49} M. Riesz, \emph{L'int{\'e}grale de Riemann-Liouville et le probl{\`e}me de Cauchy} (in French),  Acta Mathematica 81(1949), pp. 1–223.
\bibitem{Sato59} M. Sato, \emph{Theory of Hyperfunctions}, Journal of the Faculty of Science, University of Tokyo, Section 1, Mathematics, Astronomy, Physics, Chemistry (1959).
\bibitem{Zagier08}D. Zagier, \emph{Elliptic Modular Forms and Their Applications, in The 1-2-3 of Modular Forms}, Springer-Verlag (2008), pp. 181–245.
\end{thebibliography}
\end{document}